\documentclass[a4paper,11pt]{article}
\usepackage[UKenglish]{babel}
\usepackage{authblk}
\usepackage[colorlinks=true]{hyperref}
\hypersetup{citecolor=blue}
\usepackage{color}
\usepackage{amsfonts,amsthm,amssymb}
\usepackage{mathtools}
\usepackage{dsfont}
\usepackage{float}
\usepackage{epsfig}
\usepackage{graphicx}
\usepackage{caption}
\usepackage{subcaption}
\usepackage{kotex}
\usepackage{tikz}
\usepackage{pgfplots}
\usepgfplotslibrary{groupplots}
\pgfplotsset{compat=1.16}
\pgfplotsset{
    tick align=outside,
    x grid style={white},
    xmajorgrids,
    y grid style={white},
    ymajorgrids,
    axis line style={white},
    axis background/.style={fill=white!92!black},
    legend style={draw=white, fill=white},
    legend cell align={left}
}
\newtheorem{remark}{Remark}
\newtheorem{thm}{Theorem}[section]
\newtheorem{lem}[thm]{Lemma}
\newtheorem{prop}[thm]{Proposition}
\newtheorem{cor}[thm]{Corollary}
\newtheorem{defn}[thm]{Definition}

\numberwithin{equation}{section}
\numberwithin{thm}{section}
\newcommand{\beq}{\begin{equation}}
\newcommand{\eeq}{\end{equation}}
\newcommand {\f}{\frac}
\newcommand {\pa}{\partial}
\newcommand {\e}{\varepsilon}
\newcommand {\diag}{\text{diag}}
\newcommand {\init}{\text{in}}
\newcommand{\FI}{\mathds{1}}
\newcommand{\minor}[2]{[#1]_{#2}} 
\newcommand{\R}{\mathbb{R}}

\newcommand\N{{\mathbb N}}

\DeclareMathOperator{\In}{\,in\,}
\DeclareMathOperator{\tr}{\text{tr\,}}
\DeclareMathOperator{\st}{s.t.}

\newcommand\Supp{{\mathrm{supp}}\, }
\newcommand\esup{\mathop{\mathrm{ess\,sup}}}
\DeclareMathOperator{\minors}{det_2} 

\newcommand\var{\varepsilon}
\newenvironment{acknowledgment}{\noindent{\bf Acknowledgment}}{}
\title{Evolution of dietary diversity and a starvation driven cross-diffusion system as its singular limit}
\author{E. Brocchieri$^1$, L. Corrias$^1$, H. Dietert$^2$, Y.-J. Kim$^3$}
\date{\today}
\providecommand{\keywords}[1]{\small\textit{{Keywords.}} #1\\}
\providecommand{\subjclass}[1]{\small\textit{{2010 Mathematics Subject Classification.}} #1}
\begin{document}
\maketitle
\begin{abstract}
  We rigorously prove the passage from a Lotka-Volterra
  reaction-diffusion system towards a cross-diffusion system at the
  fast reaction limit. The system models a competition of two species,
  where one species has a more diverse diet than the other. The
  resulting limit gives a cross-diffusion system of a starvation
  driven type. We investigate the linear stability of homogeneous
  equilibria of those systems and rule out the possibility of
  cross-diffusion induced instability (Turing
  instability). Numerical simulations are included which are
  compatible with the theoretical results.
\end{abstract}
\section{Introduction}
\subsection{Problem setting}
We consider a semilinear reaction-diffusion system that models a competition dynamics when two species have partially different diets. The population densities of the two species are denoted by $u=u(t,x)$ and $v=v(t,x)$. The species $u$ has a more diverse diet and is divided into two substates $u_a=u_a(t,x)$ and $u_b=u_b(t,x)$ so that $u=u_a+u_b$. The system is parametrized by a small parameter $\var >0$ and written as
\begin{equation}\label{meso system}
\begin{cases}
\partial_tu_a^{\e}=d_a\Delta u_a^{\e}+f_a(u_a^\e)+\dfrac{1}{\e}Q(u_a^{\e},\,u_b^{\e},\, v^{\e}\,),\qquad&\In\,\,(0,+\infty)\times\Omega,\\
\partial_tu_b^{\e}=d_b\Delta u_b^{\e}+f_b(u_b^\e,\,v^\e)-\dfrac{1}{\e}Q(u_a^{\e},\, u_b^{\e},\,v^{\e}\,),&\In\,\,(0,+\infty)\times\Omega,\\
\partial_tv^{\e}=d_v\Delta v^{\e}+f_v(u_b^\e,\,v^\e),&\In\,\,(0,+\infty)\times\Omega,
\end{cases}
\end{equation}
where $\Omega\subset\R^\mathrm{N}$, $\mathrm{N}\ge1$, is a bounded
domain with a smooth boundary, and $d_a,d_b$ and $d_v$ are
diffusivities for the three populations. The unknown solutions depend
on the parameter $\var$ and we denote it explicitly if needed. The
above system is complemented with nonnegative initial data
\begin{equation}\label{u-v eps initial cond}
u_a^\e(0,x)=u_a^{\init}(x)\,,\quad u_b^\e(0,x)=u_b^{\init}(x)\,,\quad v^\e(0,x)=v^{\init}(x)\,,\quad x\in\Omega\,,
\end{equation}
and  zero flux boundary conditions,
\begin{equation}\label{Neumann}
d_a\nabla u_a^\e\cdot \sigma=d_b\nabla u_b^\e\cdot \sigma=d_v\nabla v^\e\cdot \sigma=0\hspace{0,1cm}, \quad\text{on }(0,+\infty)\times\,\partial\Omega\,,
\end{equation}
where $\sigma$ denotes the outward unit normal vector on the boundary $\partial\Omega$.

In this paper, we explore the effect of diet diversity in
  a competition context and show the emergence of cross-diffusion
  triggered by the different substates \(u_a\) and \(u_b\), as
  $\var\to0$. The competition dynamics is given in the reaction
terms. The reaction terms of order one are given by
\begin{align}\label{def react functions}
f_a(u_a)&:= \eta_au_a\Big(1-\dfrac{ u_a}{a}\Big),\notag\\
f_b(u_b,v)&:= \eta_bu_b\Big(1-\dfrac{u_b+v}{b}\Big),\\
f_v(u_b,v)&:= \eta_vv\Big(1-\dfrac{u_b+v}{b}\Big)\,,\notag
\end{align}
where $a,b>0$ are carrying capacities supported by two different groups of resources and $\eta_a,\eta_b$, and $\eta_v>0$ are the intrinsic growth rates of $u_a,u_b$, and $v$, respectively. The competition of the two species, $u$ and $v$, is for the resource $b$. However, the species $u$ has a diverse diet and can survive by consuming the other resource $a$ without competition. To model such a competition using a Lotka-Volterra type system, the species $u$ is divided into two substates $u_a$ and $u_b$ depending on their diets. In the above reaction terms, $u_a$ takes a logistic equation type reaction, and $u_b$ and $v$ take Lotka-Volterra competition equations type reactions as given in \eqref{def react functions}. Since competition exists only partially for the species $u$, the competition is weak for $u$. However, the species $v$ competes with $u$ for all of its resources and hence the competition is not weak in general and the competition result may depend on the parameter $\var$ (see Sections \ref{sect stability} and \ref{sect.simulation}).

The individuals of the species $u$ may freely change the type of food
depending on the availability, which is modelled by the fast reaction
term of order $\var^{-1}$,
\begin{equation}\label{def Q}
  \frac1{\var} \, Q(u_a,u_b,v)
  := \frac1{\var} \, \bigg[\phi\Big(\dfrac{u_b+v}{b}\Big)\, u_b -\psi\Big(\dfrac{u_a}{a}\Big)\,u_a\, \bigg]\,,\qquad \var>0\,.
\end{equation}
In this reaction term, $\epsilon^{-1}\,\phi\Big(\dfrac{u_b+v}{b}\Big)$
is the conversion rate for individuals in the state $u_b$ which switch
to the other state $u_a$, and
$\epsilon^{-1}\,\psi\Big(\dfrac{u_a}{a}\Big)$ is the conversion rate
in the other direction. The conversion rate
$\phi\big(\frac{u_b+v}{b}\big)$ is assumed as a function of the
starvation measure $\frac{u_b+v}{b}$ for the populations $u_b$ and
$v$. If the resource $b$ dwindles or the population $u_b+v$ increases,
the resource $b$ becomes scarce relatively, and more individuals of
population $u_b$ will convert to $u_a$ and consume the other resource
$a$. Hence, we assume that $\phi$ is an increasing function of the
starvation measure (see \cite{KK} for more discussion on the
starvation measure). In the same way, the conversion rate $\psi$ is a
function of the starvation measure $\frac{u_a}a$ for the population
$u_a$ and is assumed to be increasing. For this reason, it makes sense
to call the conversion dynamics given by \eqref{def Q} a
starvation-driven conversion, which eventually results in the
starvation-driven cross-diffusion after taking the limit $\var\to0$
(see~\cite{Choi2019,Chung2019}). More specifically, we assume the
following starvation-driven conversion hypothesis
\begin{itemize}
\item[(H1)] \label{h1} $\phi$ and $\psi$ in \eqref{def Q} are
  increasing functions belonging to $C^1([0,+\infty))$; in addition,
  there exist strictly positive constants $\delta_\psi$,
  $\delta_\phi$, $M_{\phi'}$, and $M_{\psi'}$ such that, for all
  $x\ge 0$,
  \[
    \psi(x)\ge\delta_\psi>0,\quad\phi(x)\ge\delta_\phi>0,\quad\phi'(x)\le M_{\phi'},\quad\text{and}\quad\psi'(x)\le M_{\psi'}.
  \]
\end{itemize}

The main result of the paper is that, as $\e\to0$, the (unique)  solution $(u_a^\e,u_b^\e,v^\e)$ of the initial boundary value problem \eqref{meso system}--\eqref{def Q} converges to a limit $(u_a,u_b,v)$ and this limit is a weak solution of the reaction cross-diffusion system
\begin{equation}\label{macro}
\begin{cases}
	\partial_{t}u=\Delta(d_au_a+d_bu_b)+f_a(u_a) +f_b(u_b,v),\quad&\In\,\,(0,+\infty)\times\Omega,\\
	\partial_tv=d_v\Delta v+f_v(u_b,v),&\In\,\,(0,+\infty)\times\Omega,
\end{cases}
\end{equation}
where $u_a$ and $u_b$ are (uniquely) determined by the nonlinear system
\begin{equation}\label{NLsystem}
u_a+u_b=u\quad\text{and}\quad Q(u_a,u_b,v)=0,
\end{equation}
complemented by the initial data,
\begin{equation}\label{u-v initial cond}
u(0,x)=u^{\init}(x):=u_a^{\init}(x)+u_b^{\init}(x)\,,\quad v(0,x)=v^{\init}(x)\,,\quad x\in\Omega\,,
\end{equation}
and the zero flux boundary condition,
\begin{equation}\label{BC}
\nabla (d_au_a+d_bu_b)\cdot\sigma=d_v\nabla v\cdot\sigma=0\,,\qquad\In\,\,(0,+\infty)\times\partial\Omega\,.
\end{equation}
Note that the zero flux boundary conditions in \eqref{Neumann} are equivalent to the homogeneous Neumann boundary conditions,
\[
\nabla u_a^\e\cdot \sigma=\nabla u_b^\e\cdot \sigma=\nabla v^\e\cdot \sigma=0\hspace{0,1cm}, \quad\text{on }(0,+\infty)\times\,\partial\Omega\,
\]
(see \cite{IMN} for similar diffusion operator for a single species with two phenotypes). However, after taking the singular limit, we obtain the zero flux boundary conditions \eqref{BC}, but not the homogeneous Neumann boundary conditions.

If $d_a=d_b$, the diffusion for the species $u$ given in \eqref{macro} is the homogeneous linear diffusion. However, the diffusivity of a species usually depends on its food (or prey) and $d_a\ne d_b$ in general. In that case ($d_a\ne d_b$), the diffusion for the total population in \eqref{macro} contains cross-diffusion dynamics depending on the distribution of the three populations groups, $u_a,u_b$ and $v$, through the relations in \eqref{NLsystem}. This explains the starvation-driven diffusion for the specific case of the paper, a concept formally introduced by Cho and Kim \cite{ChoKim}. Funaki \emph{et al.} \cite{Funaki} derived a macroscopic cross-diffusion model from a system of two phenotypes and a signaling chemical in the context of chemotaxis.

The proof of the convergence as $\var \to 0$ is rigorously obtained via {\it a priori} estimates for $u_a^\e,u_b^\e,$ and $v^\e$. The main tool is the energy (or entropy) functional
\begin{equation}\label{def energy E}
\mathcal{E}(u_a,u_b,v):=\int_{\Omega}h_1(u_a)\,dx+\int_{\Omega}h_2(u_b,v)\,dx\,,
\end{equation}
where
\begin{equation}\label{def energy E1 E2}
h_1(u_a):=\int_{0}^{u_a}\psi\Big(\dfrac{z}{a}\Big)\,z\,dz,\quad\text{and}\quad
h_2(u_b,v):=\int_{0}^{u_b}\phi\Big(\dfrac{z+v}{b}\Big)z\,dz\,.
\end{equation}
Notice here that the assumption (H1) implies that $h_1$ is positive,
increasing, and convex, and that $h_2$ is positive, increasing in both
variables, and convex with respect to the first variable. Therefore,
the name entropy for the function given in \eqref{def energy E} is
justified. We refer to \cite{Soresina} and \cite{DesTres} for the use
of such entropies in the context of triangular cross-diffusion systems
(that is, systems in which only one of the two equations includes a
cross-diffusion term). For more general systems, we refer to
\cite{CJ,CDJ,DDJ,DLMT,J,bothe-pierre-rolland-2012-cross} among other
works.

Then, by invoking the \textit{Aubin-Lions Lemma}, we pass to the limit along a subsequence and conclude that the limit is a weak solution of \eqref{macro}--\eqref{BC}. To use the energy estimate, we take initial values with bounded energy, which is our second hypothesis
\begin{itemize}
\item[(H2)] $u_a^{\init}\in L^1_+(\Omega)$,   $u_b^{\init}\in L^1_+(\Omega)$, $v^{\init}\in L^\infty_+(\Omega)$, and $\mathcal{E}(u^{\init}_a,u^{\init}_b,v^{\init})<\infty$.
\end{itemize}

\begin{remark}\label{rm:initial layer}
  Under Hypothesis (H2), the initial data
  $u_a^{\init},u_b^{\init},v^{\init}$ for the reaction diffusion
  system \eqref{meso system} do not satisfy a priori the nonlinear
  equation $Q(u_a^{\init},u_b^{\init},v^{\init})=0$ in
  \eqref{NLsystem}. Thus, the appearance of an initial layer is
  expected (see also Section~\ref{sect.simulation}).
\end{remark}

We conclude this introduction proposing a formal derivation of
\eqref{meso system} out of a microscopic system. We shall consider
problems left open here (such as regularity, uniqueness, stability and
long time asymptotic behaviour of the macroscopic solutions) in a
forthcoming paper, where a more general class of cross-diffusion
system is analysed.

The rest of the paper is organised as
follows. \textit{Section~\ref{sect main result}} is devoted to the
statement of the existence result. In \textit{Section~\ref{sect a
    priori estimates}}, we prove a priori estimates, which are the
preliminary ingredients for the proof of the existence result obtained
in \textit{Section~\ref{sect proof}}. The paper concludes with the
existence and linear stability analysis of trivial and non-trivial
spatially homogeneous steady states, in \textit{Section}~\ref{sect
  stability} and \textit{Appendix}~\ref{Appendix A}, with a particular
emphasis put on the coexistence state. Some numerical tests in
\textit{Section~\ref{sect.simulation}} illustrate the linear stability
analysis. The discussion in
  \textit{Section}~\ref{Discussion section} completes the article.

\subsection{Formal derivation of the reaction-diffusion system with fast switching}\label{sect micro derivation}

We explain here how the mesoscopic scale model \eqref{meso system} is
obtained at a formal level from a microscopic scale model in which the
resources inducing the competition explicitly appear.
Consider
\begin{equation}\label{micro system}
  \begin{cases}
    \partial_ts_1=\dfrac{1}{\delta }\Big[r_1s_1\Big(1-\dfrac{s_1}{A_1}\Big)-p_1s_1U_1\Big]\\[1.5ex]
    \partial_ts_2=\dfrac{1}{\delta }\Big[r_2 s_2\Big(1-\dfrac{s_2}{A_2}\Big)-p_2s_2U_2-p_Vs_2V\Big]\\[1.5ex]
    \partial_tU_1=D_1\Delta U_1+k_1p_1s_1U_1+\dfrac{1}{\e}\Big[\Phi\Big(\dfrac{p_2U_2+p_V V}{s_2}\Big)U_2 -\Psi\Big(\dfrac{p_1U_1}{s_1}\Big)U_1\Big]\\[1.5ex]
    \partial_tU_2=D_2\Delta U_2+k_2p_2s_2U_2-\dfrac{1}{\e}\Big[\Phi\Big(\dfrac{p_2U_2+p_V V}{s_2}\Big)U_2 -\Psi\Big(\dfrac{p_1U_1}{s_1}\Big)U_1\Big]\\[1.5ex]
    \partial_tV=D_V\Delta V+k_V p_V s_2V,
  \end{cases}
\end{equation}
where $\delta>0$ is the microscopic reaction time scale and $\e$ is
the mesoscopic one (hence $\delta\ll\e\ll1$). These equations describe
the time evolution of a small ecosystem with two prey population
densities (or vegetal resources), $s_1$ and $s_2$, and two predator
population densities (or harvesters of the vegetal resources), $U$ and
$V$. Moreover, the population $U$ is composed of two subpopulations
$U_1$ and $U_2$ depending on the prey they consume, i.e., $s_1$ and
$s_2$, respectively. The prey species $s_i$ follows the logistic
dynamics with a carrying capacity $A_i$ and an intrinsic growth rate
$r_i$. The predator species consume a certain amount of preys which is
proportional to the prey density with proportionality factors
$p_1,p_2$ and $p_V$. The harvested prey mass is converted to the
predator mass with conversion rates $k_1,k_2$ and $k_V$. The
subpopulations $U_1$ and $U_2$ convert to each other depending on the
availability of the prey. The two functions $\Phi$ and $\Psi$ are the
conversion rates which are respectively increasing functions of the
starvation measures $\frac{p_2U_2+p_V V}{s_2}$ and
$\frac{p_1U_1}{s_1}$. The other species $V$ consumes only the second
prey $s_2$. Hence, the active competition is only between $V$ and
$U_2$, while $U_1$ competes with $V$ passively (via
conversion). Finally, since the dispersal rate of a predator species
usually depends on the nature of its prey, $D_1\ne D_2$ in general.

\begin{remark}
  The expression \eqref{micro system} has no diffusion terms for the
  prey species \(s_1\) and \(s_2\). Since growth is the dominant
  factor for plant species and their dispersal is negligible, this is
  especially relevant when the prey species are vegetal
  resources. Mathematically, this choice yields an
    explicit form in the singular limits $\delta\to0$, see
    \eqref{system meso no adim}. Adding diffusion terms in the prey
    species equations could result in a less explicit formulas.
\end{remark}

The mescoscopic system with fixed \(\e > 0\) is obtained in the limit
\(\delta \to 0\). This is to say that the time-scale for the reaction
of the resources \(s_1\) and \(s_2\) is much faster than all other
processes. In simple predator-prey models this corresponds to a fast
dynamics of the prey which has been studied more carefully in
\cite{rinaldi-scheffer-2000-geometric-analysis-ecological-models-slow-fast-processes,
  kooi-poggiale-2018-modelling,
  poggiale-aldebert-girardot-kooi-2020-analysis}.

In this formal limit \(\delta \to 0\), we find
\[
  s_1\Big(r_1-\frac{r_1s_1}{A_1}-p_1U_1\Big)=0\ \Longrightarrow\ s_1=0\
  \mbox{ or }\ s_1=A_1\Big(1-\frac{p_1U_1}{r_1}\Big),
\]
and
\[
  s_2\Big[r_2\big(1-\frac{s_2}{A_2}\big)-p_2U_2-p_V V\Big] =0\
  \Longrightarrow\ s_2=0\
  \mbox{ or }\ s_2=A_2\Big(1-\frac{p_2U_2+p_V V}{r_2}\Big).
\]
Only the nontrivial case, $s_1\ne0\ne s_2$, is meaningful (since
$s_1=0$ and $s_2=0$ correspond to unstable equilibria), and we obtain
two relations
\[
  \frac{p_1U_1}{s_1}=\frac{r_1}{s_1}-\frac{r_1}{A_1}
  \quad\mbox{and}\quad \frac{p_2U_2+p_V V}{s_2}=\frac{r_2}{s_2}-\frac{r_2}{A_2}.
\]
Therefore, the last three equations in \eqref{micro system} turn into
\begin{equation}\label{system meso no adim}
  \begin{cases}
    \partial_t U_1=D_1\Delta U_1+A_1k_1p_1U_1\big(1-\frac{p_1U_1}{r_1}\big)+\dfrac{1}{\e}\big[\Phi U_2-\Psi U_1\big]\\[2ex]
    \partial_t U_2=D_2\Delta U_2+A_2k_2p_2U_2\,\big(1-\frac{p_2U_2+p_V V}{r_2}\big) -\dfrac{1}{\e}\big[\Phi U_2-\Psi U_1\big]\\[2.ex]
    \partial_t V\, =D_V\Delta V+ A_2k_V p_V V\,\big(1-\frac{p_2U_2+p_V V}{r_2}\big),
  \end{cases}
\end{equation}
where the conversion rates $\Phi$ and $\Psi$ read as
\[
\Phi=\Phi\Big(\frac{r_2}{s_2}-\frac{r_2}{A_2}\Big)
\quad\text{and}\quad
\Psi=\Psi\Big(\frac{r_1}{s_1}-\frac{r_1}{A_1}\Big)\,,
\]
and the Lotka-Volterra reaction dynamics of competition type naturally appears.

Now we consider the relationship between the variables in \eqref{meso system} and in \eqref{system meso no adim}. First, we define
\[
u_a^\e :=U_1,\ u_b^\e :=U_2,\ v^\e :=\frac{p_V}{p_2}V\,,
\]
and keep the same diffusivity coefficients
\[
d_a :=D_1,\ d_b :=D_2,\ d_v := D_V\,.
\]
Then, the coefficients in the Lotka-Volterra type competition dynamics, $f_a,f_b$ and $f_v$, are given as
\begin{equation}\label{relation}
\eta_a :=p_1 A_1k_1 ,\quad \eta_b :=p_2 A_2k_2,\quad \eta_v :=p_V A_2k_V,\quad a :=\dfrac{r_1}{p_1},\quad b :=\dfrac{r_2}{p_2}.
\end{equation}
Finally, the mesoscopic conversion rates are as follows
\begin{equation}\label{def:phi-psi}
\phi(x) :=\Phi\Big(\frac{r_2}{A_2}\frac{x}{1-x}\Big),\quad
\psi(x) :=\Psi\Big(\frac{r_1}{A_1}\frac{x}{1-x}\Big).
\end{equation}
After replacing variables, coefficients and functions with the above new ones, system \eqref{system meso no adim} becomes our system \eqref{meso system}.

\begin{remark}
$(i)$ The conversion rates of the microscopic model, $\Phi$ and $\Psi$, are functions of the starvation measures $\frac{p_2U_2+p_V V}{s_2}$ and $\frac{p_1U_1}{s_1}$, instead of simply $\frac{U_2+V}{s_2}$ and $\frac{U_1}{s_1}$, in order  to take into account the difference in the harvesting rates $p_2$ and $p_V$.
$(ii)$ The mesoscopic conversion rates $\phi$ and $\psi$ in \eqref{def:phi-psi} are increasing functions, since $\Phi$ and $\Psi$ are chosen to be increasing functions.
$(iii)$ It is worth noticing that the carrying capacities $a$ and $b$ for the predator species are proportional to the growth rates $r_i$'s of the prey species and that the prey carrying capacities $A_i$'s are also involved in deciding $\phi$ and $\psi$ (see \eqref{relation} and  \eqref{def:phi-psi}). $(iv)$ The macroscopic system reduces to the classical Lotka-Volterra system of competition type with linear diffusion, whenever the conversion rates $\phi$ and $\psi$ are both constant, (see the discussion section \ref{Discussion section}).
\end{remark}

\section{Statement of the main result}\label{sect main result}
Before stating our main result in Theorem \ref{thr 1} below, we introduce some notations that will be used in the sequel, and the definition of the very weak solutions of   \eqref{macro}--\eqref{BC}, with the reaction terms in \eqref{def react functions}.

We denote
\[
\begin{split}
C_c^k&\coloneqq C_c^k([0,+\infty)\times\bar{\Omega})\\
&\coloneqq\Big\{u=u(t,x):\exists\,T>0\,\st\, u\in C^k\big([0,T)\times\bar{\Omega}\big)\;\text{ and }\;\Supp u\Subset[0,T)\times\bar{\Omega}\Big\},
\end{split}
\]
and, for all $p\in [1,+\infty)$,
\[
L_{loc}^p\coloneqq L_{loc}^p((0,+\infty)\times\Omega) \coloneqq\Big\{u=u(t,x):\forall\,T>0\,, u\in L^p(\Omega_T)\Big\}\,,
\]
with $\Omega_T\coloneqq(0,T)\times\Omega$. Similarly, for $p=+\infty$,
\[
L_{loc}^\infty\coloneqq L_{loc}^\infty((0,+\infty)\times\Omega) \coloneqq\Big\{u=u(t,x):\forall\,\, T>0\,,\displaystyle\esup_{(t,x)\in\Omega_{\,T}}\arrowvert u(t,x)\arrowvert<+\infty\Big\}\,.
\]
It is worth noticing here that, due to hypothesis (H1), the function
\begin{equation}\label{def q}
q(u_b,u,v):=Q(u-u_b,u_b,v)=\phi\Big(\dfrac{u_b+v}{b}\Big)u_b -\psi\Big(\dfrac{u-u_b}{a}\Big)(u-u_b)\,,
\end{equation}
defined for $(u_b,u,v)\in[0,u]\times(0,+\infty)\times(0,+\infty)$, satisfies (for given $u>0,v>0$)
\[
\partial_{u_b}q(u_b,u,v)=\phi\Big(\dfrac{u_b+v}{b}\Big) +\frac{u_b}b\phi'\Big(\dfrac{u_b+v}{b}\Big) +\psi\Big(\dfrac{u-u_b}{a}\Big) +\frac{u-u_b}{a}\psi'\Big(\dfrac{u-u_b}{a}\Big)>0
\]
and
\[
q(0,u,v)<0\,,\qquad q(u,u,v)>0\,.
\]
Hence, for any given $(u,v)\in\R_+^2$, there exists a unique
$u_b^*(u,v)\in(0,u)$ zero of $q$, and thus a unique solution of the
nonlinear system \eqref{NLsystem} is well-defined. Furthermore, the
implicit function theorem guarantees the continuity (and even the
$C^1$ character) of $u_b^*$ with respect to $(u,v)$.

\begin{defn}\label{defweaksol}
  Let $\Omega$ be a smooth bounded domain of $\,\R^\mathrm{N}$,
  $\mathrm{N}\ge1$. Assume $u^{\init}\in L^1_+(\Omega)$, and
  $v^{\init}\in L^\infty_+(\Omega)$ be nonnegative initial
  densities. We say that the pair of nonnegative functions $(u,v)$ is
  a \textbf{very weak solution} of \eqref{macro}--\eqref{BC} over
  $(0,+\infty)\times\,\Omega$, with reaction terms \eqref{def react
    functions}, if the following conditions are satisfied
  \begin{itemize}
  \item $(u, v)$ belongs to $L_{loc}^{2}\times L_{loc}^{\infty}$,
  \item for all test functions $\xi_1,\xi_2\in C_c^2,$ with $\nabla\xi_1\cdot\sigma=\nabla\xi_2\cdot\sigma=0$ on $[0,+\infty)\times\partial\Omega$, and for $u_a, u_b$ defined as the unique solution of \eqref{NLsystem}, a.e. on $(0,+\infty)\times\,\Omega$, it holds
    \begin{align}
      -\int_{0}^{+\infty}\int_{\Omega}(\partial_{t}\xi_1)u\,dxdt\,&-\int_{\Omega}\xi_1(0,\cdot)u^{\init}dx -\int_{0}^{+\infty}\int_{\Omega}\Delta\xi_1\big(d_au_a+d_bu_b\big)dxdt\notag\\[1.5ex]
                                                                  &=\int_{0}^{+\infty}\int_{\Omega}\,\xi_1\big(f_a(u_a)+f_b(u_b,v)\big)\,dx\,dt\,,\label{vweq1}
    \end{align}
    and
    \begin{align}
      -\int_{0}^{+\infty}\int_{\Omega}\,(\partial_{t}\xi_2)\,v\,dxdt&-\int_{\Omega}\xi_2(0,\cdot\,)\,v^{\init}\,dx -d_v\int_{0}^{+\infty}\int_{\Omega}\,\Delta\xi_2\,v\,dxdt\notag\\[1.5ex]
                                                                    &=\int_{0}^{+\infty}\int_{\Omega}\,\xi_2\,f_v(u_b,v)\,dxdt\,.\label{vweq2}
    \end{align}
  \end{itemize}
\end{defn}
\medskip

We observe that all terms in \eqref{vweq1}--\eqref{vweq2} are well-defined thanks to the assumptions (H2) on the initial densities $u^\init,v^\init$,  to the $L^2$ integrability of the sub-population densities $u_a,u_b$ and  to the $L^{\infty}$ bound for $v$. Remember that the logistic structure of the reaction functions $f_a,f_b,f_v$ involves at most quadratic nonlinearities.
\medskip

\begin{thm}\label{thr 1}
Let $\Omega$ be a smooth bounded domain of $\,\R^\mathrm{N}$, $\mathrm{N}\ge1$. Assume (H1) and (H2) on parameters and initial data $u_a^{\init},u_b^{\init}$,
$v^{\init}$, respectively. We denote $(u_a^\e,u_b^\e,v^\e)$ the unique global strong (for $t>0$) solution of system \eqref{meso system}--\eqref{Neumann} with those initial data. Then, the triplet $(u_a^\e,u_b^\e,v^\e)$ converges a.e. $(t,x)\in(0,+\infty)\times\,\Omega$ (up to extraction of a subsequence) towards a nonnegative triplet $(u_a,u_b,v)$, as $\e\rightarrow 0$. Moreover, the functions $u := u_a +u_b$, $v$ satisfy the nonlinear system \eqref{NLsystem}, for a.e. $(t,x)\in(0,+\infty)\times\Omega$, and the following bounds:
$u \in L^q(\Omega_T)$ for $q = 2+ 2/\mathrm{N}$ if $\mathrm{N}\ge 3$, $q<3$ if $\mathrm{N}=2$ and $q=3$ if $\mathrm{N}=1$;
$v\in L^{\infty}(\Omega_T)$; $|\nabla u| \in L^2(\Omega_T)$; and for the same previous $q$, $|\nabla v| \in L^{2q}(\Omega_T)$;
$\pa_{x_i,x_j}v\,,\pa_t  v\in L^q(\Omega_T)$, $i,j=1,\dots,\mathrm{N}$. Finally, $(u,v)$ is a very weak solution of the macroscopic system \eqref{macro}--\eqref{BC} with the reaction terms \eqref{def react functions}, in the sense of \textit{Definition \ref{defweaksol}}.
\end{thm}

\section{Proof of the main Theorem}

We first recall that  for any $\e >0$, there exists a unique global strong (for $t>0$) solution $(u_a^\e,u_b^\e,v^\e)$ solution to system \eqref{meso system}--\eqref{Neumann}, under the assumption on the initial data of Theorem \ref{thr 1}. We refer for example to \cite{LD,QS} for obtaining such a result.
\par

\subsection{A priori estimates}\label{sect a priori estimates}

In this section we shall obtain {\it{a priori}} estimates on the subpopulation densities $u_a^\e, u_b^\e$, on the total population densities $u^\e\coloneqq u_a^\e+u_b^\e$ and $v^\e$, and on $Q(u_a^\e, u_b^\e,v^\e)$. More specifically, we take advantage of the triangular structure of the system that give us {\it{a priori}} estimates on the density $v^\e$ and its derivatives (see Lemma \ref{Lemma 1}). The reaction functions $f_a$ and $f_b$ of competition type allow us to control the total mass $\int_\Omega u^\e(t)\,dx$, and to get an $L^2(\Omega_T)$ estimate on $u^\e$ (see Lemma \ref{Lemma 2}).  The latter will be employed in Lemma \ref{Lemma 3} to obtain estimates on $\nabla u_a^\e$, $\nabla u_b^\e$ and $Q(u_a^\e, u_b^\e,v^\e)$, through the use of the energy functional \eqref{def energy E}--\eqref{def energy E1 E2}. In addition, the triplet $(u_a^\e, u_b^\e,v^\e)$ will be shown to have finite energy $\mathcal{E}(T)$ as well, for all $T>0$.
\medskip

Hereafter, all constants $C$ and $C_T$ are strictly positive and may depend  on $\Omega$, the initial data $u_a^{\init},u_b^{\init},v^{\init}$, the coefficients in system \eqref{meso system}, the transition functions $\phi,\psi$ and on $T$, but never on $\e$. They may change also from line to line in the computations.

\begin{lem}\label{Lemma 1}
  Under the hypothesis of Theorem~\ref{thr 1}, the following
  statements hold:
\begin{itemize}
\item[(i)] there exists a constant $C>0$ such that for all $\e>0$
\begin{equation}\label{eq1 lemma v}
\| v^\e\|_{L^{\infty}((0,+\infty)\times\Omega)}\leqslant C\,;
\end{equation}
\item[(ii)] for all $q\in(1,+\infty)$ there exists a constant $C(q)>0$ such that, for all $\e>0$, $T>0$ and all $i, j =1,..,\mathrm{N}$,
\begin{equation}\label{eq2 lemma v}
\|\partial_{t}v^\e\|_{L^q(\Omega_T)}+\|\partial_{x_i,x_j} v^\e\|_{L^{^{q}}(\Omega_T)}\leqslant C(q)(1+\| u_b^\e\|_{L^q(\Omega_T)})\,;
\end{equation}
\item[(iii)] for all $q\in(1,+\infty)$ there exist $C(q,\mathrm{N})>0$ and $C(q)>0$ such that, for all $\e>0$ and all $T>0$,
\begin{equation}\label{eq3 lemma v}
\|\nabla v^\e\|_{L^{2q}(\Omega_T)}^{2q}\leq C(q,\mathrm{N})(1+\| u_b^\e\|^q_{L^q(\Omega_T)})+C(q)\,T.
\end{equation}
\end{itemize}
\end{lem}
\begin{remark}\label{rmk choice of q}\hfill\\
In the sequel, the value of $q$ in \eqref{eq2 lemma v}--\eqref{eq3 lemma v} will be first chosen equal to 2 (see Lemma~\ref{Lemma 2}), and then to a different number after Corollary \ref{corollary estimates}.
\end{remark}
\begin{proof}
It is easily seen that
\begin{equation}\label{max principle}
0\leqslant v^\e (t,x)\leqslant\,K\coloneqq\max\big\{\,\Arrowvert v^{\init}\Arrowvert_{L^{\infty}(\Omega)}\,;\,b\big\},\quad\mbox{for a.e.}\hspace{0,3cm}(t,x)\,\in\,(0,+\infty)\times\Omega\,.
\end{equation}
Indeed, by the existence result of strong solution for \eqref{meso system}, we know that the nonnegativity of $v^\e$ is preserved in time. Concerning the upper bound in \eqref{max principle}, it is obtained by multiplying the equation for $v^\e$ in \eqref{meso system} by $(v^\e-K)^+:=\max\{0,v^\e-K\}$ and integrating over $\Omega$,  to obtain for all $t>0$,
\[
\int_\Omega(v^\e(t)-K)_+^2\,dx\le \int_\Omega(v^{\init,\e}-K)_+^2\,dx=0\,.
\]

Next, by the maximal regularity property of the heat equation (see \cite{Lamberton} and the references therein), for all $q\in(1,+\infty)$ there exists a strictly
positive constant $C$, which only depends on $\Omega$ and $q$, such that for all $i, j =1,..,\mathrm{N}$,
\begin{align}
\Arrowvert\partial_{t}v^\e\Arrowvert_{L^q(\Omega_T)}+\|\partial_{x_i,x_j} v^\e\|_{L^{^{q}}(\Omega_T)}&\le C(\Arrowvert f_v(u_b^\e,v^\e)\Arrowvert_{L^{^{q}}(\Omega_T)}+\Arrowvert v^{\init}\Arrowvert_{L^{^{q}}(\Omega)}\big)\notag\\
&\le C\big(1+\Arrowvert u_b^\e\Arrowvert_{L^q(\Omega_T)}\big)\label{eq4 lemma v},
\end{align}
so that estimate~\eqref{eq2 lemma v} holds. Then, thanks to the Gagliardo-Nirenberg inequality \cite{Nirenberg}, for all $q\in(1,+\infty)$, there exists $C(q)>0$, such that, for all $t>0$ and $i=1,\dots \mathrm{N}$, we have
\[
\|\partial_{x_i}v^\e(t)\|_{L^{2q}(\Omega)}\leq C(q)\sum\limits_{j=1}^{\mathrm{N}}\| \partial_{x_i,x_j}v^\e(t)\|_{L^q(\Omega)}^{1/2}\,\| v^\e(t)\|_{L^{\infty}(\Omega)}^{1/2}+C(q)\Arrowvert v^\e(t)\|_{L^{\infty}(\Omega)}\,.
\]
Integrating the above inequality over $(0,T)$ and using \eqref{eq1 lemma v} and \eqref{eq4 lemma v}, we get estimate \eqref{eq3 lemma v}.
\end{proof}
\begin{lem}\label{Lemma 2}
Under the hypothesis of Theorem \ref{thr 1}, for all $T>0$ there exists $C_T>0$ such that for all $\e>0$ the following estimates hold:
\begin{equation}\label{eq lemma u}
\sup_{t\,\in\,[0,T]}\,\int_{\Omega}(u_a^\e+u_b^\e)(t)\,dx\le C_{T}\,\hspace{0,6cm}\mbox{ and }\hspace{0,6cm}\,\Arrowvert u_a^\e+u_b^\e\Arrowvert_{L^{^2}(\Omega_T)}\le C_{T}\,.
\end{equation}
\end{lem}
\begin{proof}
  Adding the first two equations in \eqref{meso system} and using the
  positivity of $u^\e_a,u^\e_b,v^\e$, we get
  \begin{align}
    \partial_{t}(u_a^\e+u_b^\e)
    &\le\,d_a\Delta u_a^\e+d_b\Delta u_b^\e
      +\eta_au_a^\e\left(1-\dfrac{u_a^\e}{a}\right)
      +\eta_bu_b^\e\left(1-\dfrac{u_b^\e}{b}\right)\label{disug1 in eq ua ub}\\
    &\le\,d_a\Delta u_a^\e+d_b\Delta u_b^\e+\dfrac{1}{4}\left(a\eta_a+b\eta_b\right).\label{ineq eq Ua Ub}
  \end{align}
  Then, integrating \eqref{ineq eq Ua Ub} over $\Omega$, the inequality becomes
  \[
    \dfrac{d}{dt}\int_{\Omega}\big(u_a^\e+u_b^\e \big)(t)\,dx\le C\,,
  \]
  implying, for all $t$ in $[0,T]$, that
  \begin{equation}
    \Arrowvert u_a^\e(t)+u_b^\e(t)\Arrowvert_{L^1(\Omega)}\le\Arrowvert u^{\init}_a+u^{\init}_b\Arrowvert_{L^1(\Omega)}+C\,T\,.
  \end{equation}

  In order to obtain the $L^2(\Omega_T)$ estimate for $u_a^\e+u_b^\e$, we integrate inequality \eqref{disug1 in eq ua ub} first over $\Omega$ and then over $(0,t)$, for $t\in(0,T)$, to obtain
  \begin{align*}
    \int_{\Omega}(u_a^\e+u_b^\e)(t)\,dx+\dfrac{\eta_a}{a}\displaystyle\int_{\Omega_t}(u_a^\e)^{^{2}}\,dx\,dt& +\frac{\eta_b}{b}\displaystyle\int_{\Omega_t}(u_b^\e)^{^{2}}\,dx\,dt\\
                                                                                                            &\le\Arrowvert u_a^{\init}+u_b^{\init}\Arrowvert_{L^1(\Omega)}+C\| u_a^\e+u_b^\e\|_{L^1(\Omega_T)}\,.
  \end{align*}
  The second estimate in \eqref{eq lemma u} follows, using the first one.
\end{proof}

\begin{lem}\label{Lemma 3}
  Under the hypothesis of Theorem \ref{thr 1}, for all $T>0$, there
  exists $C_T>0$ such that, for all $\e>0$, the global solution of
  \eqref{meso system} satisfies
  \begin{equation}
    \label{eq energy 1 lemma}
    \begin{split}
      \mathcal{E}(u_a^\e,u_b^\e,v^\e)(T)
      &+ \| \nabla u_a^\e \|_{L^2(\Omega_T)}^2
      + \| \nabla u_b^\e \|_{L^2(\Omega_T)}^2
      + \frac{1}{\e}
      \| Q(u_a^\e,u_b^\e,v^\e) \|_{L^2(\Omega_T)}^2 \\
      &\le C_T\,.
    \end{split}
  \end{equation}
\end{lem}
\begin{proof}
  We shall analyse the evolution of $\mathcal{E}$, along the
  trajectories of the solution of \eqref{meso system}. Thus, from the
  first equation in \eqref{meso system} and assumption (H1), we have
  \begin{equation}
    \label{dtE1}
    \begin{split}
      \dfrac{d}{dt}\int_{\Omega}h_1(u_a^\e)\,dx
      &=\int_{\Omega}(\partial_{t}u_a^\e)\, u_a^\e\,\psi\Big(\dfrac{u^\e_a}{a}\Big)\,dx\\
      &=-d_a\int_{\Omega}\Big[\psi\Big(\dfrac{u^\e_a}{a}\Big) +\dfrac{u^\e_a}{a}\,\psi'\Big(\dfrac{u^\e_a}{a}\Big)\Big]|{\nabla u_a^\e}|^2\,dx\\
      &\quad+\int_{\Omega}\,u_a^\e\,f_a(u_a^\e)\,\psi\Big(\dfrac{u^\e_a}{a}\Big)\,dx +\dfrac1\e\int_{\Omega}u_a^\e\,\psi\Big(\dfrac{u^\e_a}{a}\Big)\,Q_\e\,dx\\
      &\le -d_a\delta_\psi\int_{\Omega}|{\nabla u_a^\e}|^2\,dx\\
      &\quad+C\int_\Omega(u_a^\e)^2\Big(1-\frac{u_a^\e}{a}\Big)\FI_{\{u_a^\e\le a\}}dx+\dfrac1\e\int_{\Omega}u_a^\e\psi\Big(\dfrac{u^\e_a}{a}\Big)Q^\e dx.
    \end{split}
  \end{equation}

  Concerning the second term in the energy \eqref{def energy E}, we see that
  \begin{equation}
    \label{est:E2}
    \begin{split}
      &\hspace{-1cm} \dfrac{d}{dt}\int_{\Omega}h_2(u_b^\e,v^\e)\,dx\\
      &=\int_{\Omega}(\partial_{t}u_b^\e)\, u_b^\e\,\phi\Big(\frac{\,u^\e_b+v^\e\,}{b}\Big)dx
      +\int_{\Omega}(\partial_{t}v^\e)\partial_vh_2(u_b^\e,v^\e)\,dx \\
      &\eqqcolon \,I_1+I_2\,.
    \end{split}
  \end{equation}
  Using the second equation in \eqref{meso system}, $I_1$ rewrites as follows
  \begin{equation}\label{est:I1}
    \begin{split}
      I_1&\le-d_b\int_\Omega|\nabla u_b^\e|^2\left[\phi\left(\f{u_b^\e+v^\e}b\right)+\f{u_b^\e}b\,\phi'\left(\f{u_b^\e+v^\e}b\right)\right]\,dx\\
      &\quad-d_b\int_\Omega \f{u_b^\e}b\,\phi'\left(\f{u_b^\e+v^\e}b\right)\nabla u_b^\e\cdot\nabla v^\e\,dx\\
      &\quad+C\int_\Omega\,(u_b^\e)^2\left(1-\f{u_b^\e+v^\e}{b}\right)\FI_{\{u_b^\e+v^\e\le
        b\}}\,dx \\
      &\quad-\f 1 \e\int_\Omega u_b^\e\,\phi\left(\f{u_b^\e+v^\e}b\right)\,Q^\e\,dx\,.
    \end{split}
  \end{equation}
  On the other hand, observing that
  \begin{equation}
    \label{devh2}
    \partial_vh_2(u_b,v)
    =
    \int_0^{u_b}\f zb\,
    \phi'\left(\f{z+v}b\right)\,dz
    = u_b\phi\left(\f{u_b+v}b\right)-\int_0^{u_b}\phi\left(\f{z+v}b\right)\,dz\,,
  \end{equation}
  the positivity of $\phi'$ implies that $\partial_vh_2$ is positive and
  \[
    \begin{split}
      &\int_{\Omega}\partial_vh_2(u_b^\e,v^\e)\,f_v(u_b^\e,v^\e)\,dx \\
      &\le\eta_v\int_{\Omega}\partial_vh_2(u_b^\e,v^\e)\, v^\e
      \left(1-\f{u_b^\e +v^\e}{b}\right)\FI_{\{u_b^\e+v^\e\le b\}}\,dx\\
      &\le\eta_v\int_{\Omega}u_b^\e
      \phi\left(\f{u_b^\e+v^\e}b\right)\,v^\e
      \left(1-\f{u_b^\e+v^\e}{b}\right)\FI_{\{u_b^\e+v^\e\le b\}}\,dx\,.
    \end{split}
  \]
  Therefore, we obtain
  \begin{equation}\label{est:I2}
    \begin{split}
      I_2
      &\le -d_v\int_\Omega\partial_{vv}h_2(u_b^\e,v^\e)\,
      |\nabla v^\e|^2\,dx
      -d_v\int_\Omega\partial_{vu_b}h_2(u_b^\e,v^\e)\,
      \nabla u_b^\e\cdot\nabla v^\e\,dx\\
      &\quad+\eta_v\int_{\Omega}u_b^\e\phi\left(\f{u_b^\e+v^\e}b\right)\,
      v^\e\left(1-\f{u_b^\e+v^\e}{b}\right)\FI_{\{u_b^\e+v^\e\le b\}}\,dx\,.
    \end{split}
  \end{equation}
  Computing from \eqref{devh2}
  \[
    \partial_{vu_b}h_2(u_b,v)=\f{u_b}b\,\phi'\left(\f{u_b+v}b\right)\,,
  \]
  and plugging estimates \eqref{est:I1} and \eqref{est:I2} into \eqref{est:E2}, we end up with the estimate
  \begin{equation}\label{eq:dt h2}
    \begin{split}
      \dfrac{d}{dt}\int_\Omega h_2(u_b^\e,v^\e)\,dx\le&
      -d_b\int_\Omega
      \left[\phi\left(\f{u_b^\e+v^\e}b\right)+\f{u_b^\e}b\,\phi'\left(\f{u_b^\e+v^\e}b\right)\right]\,
      |\nabla u_b^\e|^2\,dx\\
      &-d_v\int_\Omega\partial_{vv}h_2(u_b^\e,v^\e)|\nabla v^\e|^2\,dx\\
      &-(d_b+d_v)\int_\Omega
      \f{u_b^\e}b\,\phi'\left(\f{u_b^\e+v^\e}b\right) \nabla u_b^\e\cdot\nabla v^\e\,dx\\
      &+C\int_\Omega\,(u_b^\e)^2\left(1-\f{u_b^\e+v^\e}{b}\right)\FI_{\{u_b^\e+v^\e\le b\}}\,dx\\
      &+\eta_v\int_{\Omega}u_b^\e\phi\left(\f{u_b^\e+v^\e}b\right)\,
      v^\e\left(1-\f{u_b^\e+v^\e}{b}\right)\FI_{\{u_b^\e+v^\e\le b\}}\,dx\\
      &-\f1\e\int_\Omega u_b^\e\,\phi\left(\f{u_b^\e+v^\e}b\right)\,Q^\e\,dx\,.
    \end{split}
  \end{equation}
  Next, using the positivity of $\phi'$ again, we estimate the third term in \eqref{eq:dt h2} with a weight \(\eta>0\) as
  \begin{align*}
    &-(d_b{+}d_v)\int_\Omega \f{u_b^\e}b\,\phi'\left(\f{u_b^\e+v^\e}b\right)\nabla
      u_b^\e\cdot\nabla v^\e\,dx\\
    &\le (d_b{+}d_v)\f{\eta}2\int_{\Omega}\f{u_b^\e}b\,\phi'\left(\f{u_b^\e{+}v^\e}b\right)|\nabla u_b^\e|^2\,dx
      +\f{d_b{+}d_v}{2\eta}\int_{\Omega}\f{u_b^\e}b\,\phi'\left(\f{u_b^\e{+}v^\e}b\right)|\nabla v^\e|^2\,dx.
  \end{align*}
  Thus, choosing $\eta\in(0,2d_b(d_b+d_v)^{-1})$, gives $C(\eta):=(d_b-(d_b+d_v)\frac\eta2)>0$, and inequality \eqref{eq:dt h2} becomes
  \begin{equation}\label{dtE2}
    \begin{split}
      \dfrac{d}{dt}\int_\Omega h_2(u_b^\e,v^\e)\,dx
      &\le-d_b\delta_\phi\int_\Omega|\nabla u_b^\e|^2\,dx
      -d_v\int_\Omega\partial_{vv}h_2(u_b^\e,v^\e)|\nabla v^\e|^2\,dx\\
      &\quad-C(\eta)\int_{\Omega}\f{u_b^\e}b\,\phi'\left(\f{u_b^\e+v^\e}b\right)|\nabla u_b^\e|^2\,dx\\
      &\quad+\f{(d_b+d_v)}{2\eta}\int_\Omega\f{u_b^\e}b\,\phi'\left(\f{u_b^\e+v^\e}b\right)|\nabla v^\e|^2\,dx\\
      &\quad+C-\f1\e\int_\Omega u_b^\e\,\phi\left(\f{u_b^\e+v^\e}b\right)\,Q^\e\,dx\,.
    \end{split}
  \end{equation}
  Finally, by assumption (H1), the derivative
  \[
    \begin{split}
      \partial_{vv}h_2(u_b,v)
      &= \f{u_b}b\,\phi'\left(\f{u_b+v}b\right)
      - \left[\phi\left(\f{u_b+v}b\right)-\phi\left(\f{v}b\right)\right] \\
      &= \int_v^{u_b+v} \left[\phi'\left(\f{u_b+v}b\right)-\phi'\left(\f zb\right)\right]\,\f{dz}b
    \end{split}
  \]
  satisfies
  \[
    |\partial_{vv}h_2(u_b,v)|\le 2\,M_{\phi'}\f{u_b}b\,.
  \]
  Therefore, adding \eqref{dtE1} and \eqref{dtE2}, and using the
  boundedness of $\phi'$ again, we arrive at the following estimate for
  the time derivative of the energy
\begin{align}
  \dfrac{d}{dt}\mathcal{E}(u_a^\e(t),u_b^\e(t),v^\e(t))
  &\le
    -d_a\delta_\psi\int_{\Omega}|{\nabla u_a^\e}|^2\,dx-d_b\delta_\phi\int_\Omega|\nabla u_b^\e|^2\,dx\notag\\
  &\quad - C(\eta)\int_{\Omega}\f{u_b^\e}b\,\phi'\left(\f{u_b^\e+v^\e}b\right)|\nabla u_b^\e|^2\,dx\label{dtE}\\
  &\quad + C\Arrowvert u_a^\e+u_b^\e\Arrowvert_{L^2(\Omega)}\Arrowvert \nabla v^\e\Arrowvert_{L^4(\Omega)}^2
    -\dfrac1\e\int_{\Omega}(Q_\e)^2dx+C.\notag
\end{align}
Integrating in time over $[0,\,T]$ the latter inequality, estimate \eqref{eq energy 1 lemma} is proved by the means of Lemma \ref{Lemma 1} (with $q=2$), Lemma \ref{Lemma 2} and the boundedness of the initial energy.
\end{proof}

We conclude this section by giving improved estimates from interpolation arguments.
\begin{cor}\label{corollary estimates}
Under the hypothesis of Theorem \ref{thr 1}, for all $T>0$, the following estimates hold:
\begin{equation}\label{eq corollary}
\Arrowvert u^\e_a+u^\e_b\Arrowvert_{L^2\,\big([0,T];\,H^1(\Omega)\big)}\leqslant C_T\,,
\end{equation}
and
\begin{equation}\label{space interpol}
\Arrowvert u^\e_a+u^\e_b\Arrowvert_{ L^{q}(\Omega_T) }\leqslant C_T\,,
\end{equation}
where
\begin{equation}\label{space interpol2}
q :=
\begin{cases}
2+2/\mathrm{N} &\mbox{if }\,\mathrm{N}>2,\\
3,\hspace{1,5cm}&\mbox{if }\,\mathrm{N}=1.
\end{cases}
\end{equation}
and $q<3$ if $\mathrm{N}=2$.
\end{cor}
\begin{proof}
The following argument is performed for the subpopulation $u^\e_a$. It can be applied similarly to $u^\e_b$ and thus to $u^\e_a+u^\e_b$.

Lemma \ref{Lemma 2} and \ref{Lemma 3} give that $u_a^\e$ is bounded in $L^2([0,T];\,H^1(\Omega))$. Thus, by the Sobolev embedding theorem, we have that $u_a^\e$ is bounded in $L^2([0,T];\,L^{\mathrm{N}^*}(\Omega))$, with $\mathrm{N}^*=\frac{2\mathrm{N}}{\mathrm{N}-2}$ if $\mathrm{N}>2$, $\mathrm{N}^*\in[2,+\infty)$ if $\mathrm{N}=2$ and $\mathrm{N}^*=\infty$ if $\mathrm{N}=1$. Since we also know that $u_a^\e$ is bounded in $L^\infty([0,T];\,L^1(\Omega))$, by interpolation we obtain that $u_a^\e$ is bounded in $L^q(\Omega_T)$, with $q$ as in \eqref{space interpol2}.
\end{proof}

\begin{remark}
At this point, using Lemma \ref{Lemma 1} again, we see that $\pa_t v^\e$ and $\nabla\nabla v^\e$ are bounded in $L^q(\Omega_T)$.
\end{remark}

\subsection{End of the proof of the main result}\label{sect proof}

\begin{proof}[End of the proof of Theorem \ref{thr 1}]

The proof is divided in four steps and uses compactness to identify limits along subsequences. The first and the second one focus on the identification of the limit (as $\e\to0$) of the densities $v^\e$ and $u^\e=u_a^\e+u_b^\e$, a.e. in $[0,T]\times\Omega$, respectively. In the third step we obtain the a.e. convergence of the subpopulation densities $u^\e_a,u^\e_b$ and we identify the obtained limit as the unique solution of the nonlinear system \eqref{NLsystem}. The convergence argument is also extended globally in time by a diagonal  argument. Finally, the proof is concluded in the fourth step, taking the limit as $\e$ tends to  zero, in the very weak formulation of the system satisfied by $u^\e=u_a^\e+u_b^\e$ and $v^\e$.\\

\noindent\textit{First step.} Let $T>0$ be arbitrarily fixed. Thanks to the control of the density $v^\e$ given in Lemma \ref{Lemma 1} and to the boundedness of $u_a^\e+u_b^\e$ in $L^2(\Omega_T)$ obtained in Lemma \ref{Lemma 2}, we have that $(v^\e)_\e$ is bounded in $L^4([0,T];W^{1,4}(\Omega))$ and $(\partial_t v^\e)_\e$ is bounded in $L^2([0,T];L^2(\Omega))$. Therefore, by applying Rellich's Theorem,
 there exists a subsequence, still denoted $v^\e$, and $v\in L^4(\Omega_T)$ such that, as $\e\to0$,
\begin{equation}\label{a.e. v}
v^\e(t,x)\longrightarrow v(t,x)\,,\qquad\text{a.e. on } [0,T]\times\Omega\,.
\end{equation}
Moreover,
\begin{equation}\label{weak conv grad v}
\nabla v^\e \rightharpoonup \nabla v\, \,\qquad\text{in } L^4(\Omega_T),
\end{equation}
and due to Lemma \ref{Lemma 1} again, $v$ is nonnegative and belongs to $ L^\infty(\Omega_T),$ while $\nabla v$ lies in $ L^4(\Omega_T)$.\\

\noindent\textit{Second step.}
We rewrite the parabolic equation satisfied by the density $u^\e= u_a^\e+u_b^\e$ as
\begin{equation}\label{parabolic eq ua+ub}
\partial_t u^\e=\Delta(d_a\,u_a^\e+d_b\,u_b^\e)+f_a(u_a^\e)+f_b(u_b^\e,v^\e)\,.
\end{equation}

Thanks to Corollary \ref{corollary estimates}, we see that $(u^\e)_\e$ is uniformly bounded in $L^2([0,T];H^1(\Omega))$ and in $L^{2+2\delta}(\Omega_T)$ for some $\delta >0$,
so that the reaction term in \eqref{parabolic eq ua+ub} is uniformly bounded in $L^{1+\delta}(\Omega_T)$. Then $(\pa_t(u_a^\var + u_b^\var))_{\var}$ is uniformly bounded
in
\par\noindent
 $L^{1+\delta}([0,T]; W^{-1,1+\delta}(\Omega))$. Thus, Aubin-Lions' lemma (cf. \cite{Moussa}) yields a subsequence (still denoted $u^\e$), and a
function $u\geq 0,\, u\in L^2(\Omega_T)$, such that, as $\e\to0$,
\begin{equation}\label{ae conv u}
u^\e(t,x)=u_a^\e(t,x)+u_b^\e(t,x)\longrightarrow u(t,x)\,,\qquad\text{a. e. in }\Omega_T\,,
\end{equation}
where the nonnegativity of $u$ follows from that of $u^\e$. Furthermore,
\begin{equation}
\nabla u^\e\,\rightharpoonup \, \nabla u\,\qquad\text{in } L^2(\Omega_T)\,,
\end{equation}
and
\begin{align*}
\Arrowvert u \,\Arrowvert_{L^{^{2}}(\Omega_T)}& = \lim_{\e\rightarrow\,0}\,\Arrowvert\,u_a^\e+u_b^\e\,\Arrowvert_{L^{^2}(\Omega_T)}\leqslant C_T\,,\\
\Arrowvert\nabla u\,\Arrowvert_{L^{^2}(\Omega_T)}&\le \liminf_{\e\rightarrow\,0}\,\Arrowvert \,\nabla u^\e\,\Arrowvert_{L^{^2}(\Omega_T)}\leqslant C_T\,.
\end{align*}

\noindent\textit{Third step.}
The energy estimate \eqref{eq energy 1 lemma} yields the estimate
\begin{equation}\label{ineq sqrt eps}
\Big\Arrowvert\phi\Big(\dfrac{u^\e_b+v^\e}b\Big)u_b^\e-\psi\Big(\dfrac{u^\e_a}{a}\Big) u_a^\e\Big\Arrowvert_{L^2(\Omega_T)}\le\sqrt{\e}\, C_T\,.
\end{equation}
Therefore, $Q(u_a^\e,u_b^\e,v^\e)$ converges to zero in $L^2(\Omega_T)$, as $\e\rightarrow 0$, and (up to extraction of a subsequence)
\begin{equation}\label{a.e. Q in omegaT}
\phi\Big(\dfrac{u_b^\e+v^\e}{b}\Big)u_b^\e-\psi\Big(\dfrac{u_a^\e}{a}\Big)u_a^\e\, \longrightarrow\,0,\hspace{1cm}\mbox{ a.e. in }\,\,\Omega_{\,T}\,.
\end{equation}

It remains to prove the existence of the a.e. limit of subsequences of $(u^\e_a)_\e,(u^\e_b)_\e$ and to obtain that this limit is, a.e.\,over $\Omega_T$, the unique solution of \eqref{NLsystem}, corresponding to the functions $u$ and $v$  obtained in \eqref{ae conv u} and \eqref{a.e. v}, respectively.

Let us denote $\big(u_a^*(u,v),u_b^*(u,v)\big)$ the unique solution of
\eqref{NLsystem}. Then, using the function $q$ defined in \eqref{def
  q}, we get
\[
  \begin{split}
    Q(u_a^\e,u_b^\e,v^\e)&=Q(u_a^\e,u_b^\e,v^\e)-Q(u_a^*(u^\e,v^\e),u_b^*(u^\e,v^\e),v^\e)\\
    &=q(u_b^\e,u^\e,v^\e)-q(u_b^*(u^\e,v^\e),u^\e,v^\e)\\
    &= \partial_{u_b}q(\zeta,u^\e,v^\e)\,(u_b^\e-u_b^*(u^\e,v^\e))\,,
  \end{split}
\]
for some intermediate value \(\zeta\) between \(u_b^\e\) and \(u_b^*(u^\e,v^\e)\). Hence by hypothesis (H1) we obtain
\[
|Q(u_a^\e,u_b^\e,v^\e)|\ge (\delta_\phi+\delta_\psi)|u_b^\e-u_b^*(u^\e,v^\e)|\,.
\]
Thus by \eqref{a.e. Q in omegaT}, $|u_b^\e-u_b^*(u^\e,v^\e)|\to0$ as $\e\to0$, a.e. in $\Omega_T$. Finally, the proved convergence \eqref{ae conv u} and \eqref{a.e. v} and the continuity of $u_b^*$ with respect to its arguments, yields the desired result, i.e.,
\[
u_b^\e\to u_b^*(u,v)\,,\quad u_a^\e=u^\e-u_b^\e\to u_a^*(u,v)\,,\quad \e\to0\,,\qquad \text{ a.e. in }\,\,\Omega_{\,T}\,.
\]

To conclude, let us remark that all the a.e. convergence results
obtained so far have been performed on $[0,T]$, for any arbitrary
$T>0$. Since $(u_a^\e,u_b^\e,v^\e)$ is defined on $[0, +\infty)$, by
extracting subsequences, these arguments can be replicated in the time
intervals $[0,2T]$, $[0,3T]$, and so on. Then by Cantor's diagonal
argument, the convergences \eqref{a.e. v}, \eqref{ae conv u} and
\eqref{a.e. Q in omegaT}, and the convergence
of the pair $(u^\e_a,u^\e_b)$ towards the solution of \eqref{NLsystem} are verified a.e. in $(0,+\infty)\times\Omega$. \\

\noindent\textit{Fourth step.}
We shall prove now that $(u,v)$ is a weak solution of \eqref{macro}, in the sense of Definition \ref{defweaksol}. For this purpose, let us consider two test functions $\xi_1,\xi_2$ in $C_c^2$, satisfying $\nabla\xi_1\cdot\sigma=\nabla\xi_2\cdot\sigma=0$, on $[0,T]\times\partial\Omega$. Multiplying the equation satisfied by $u_a^\e+u_b^\e$ by $\xi_1$ and the third equation of \eqref{meso system} by $\xi_2$ and integrating over $(0,+\infty)\times\Omega$, we get,
\begin{equation}
  \label{very weak eq xi}
  \begin{split}
    &-\int_{0}^{\infty}\int_{\Omega}\,(\partial_{t}\xi_1)\,(u_a^\e+u_b^\e)\,dx\,dt -\int_{\Omega}\xi_1(0)\,\big(u_a^{\init,\e}+u_b^{\init,\e}\big)\,dx=\\
    &\int_{0}^{\infty}\int_{\Omega}\,\Delta\xi_1\, \big(\,d_au_a^\e+d_bu_b^\e\,\big)\,dx\,dt
    +\int_{0}^{\infty}\int_{\Omega}\,\xi_1\big(f_a(u_a^\e)+f_b(u_b^\e,v^\e)\big)\,dx\,dt\,,
  \end{split}
\end{equation}
and
\begin{equation}\label{very weak eq x2}
  \begin{split}
    -\int_{0}^{\infty}\int_{\Omega}\,(\partial_t\xi_2\,)\,v^\e\,dx&\,dt-\int_{\Omega}\xi_2(0)\,{v}^{\init,\e}\,dx=\\
    &d_v\int_{0}^{\infty}\int_{\Omega}\,\Delta\xi_2\, v^\e\,dx\,dt+\int_0^\infty\int_{\Omega}\,\xi_2\,f_v(u_b^\e,v^\e)\,dx\,dt\,.
  \end{split}
\end{equation}

Concerning the equation \eqref{very weak eq xi}, the convergence results obtained in the previous steps and the estimates in \eqref{eq lemma u} allow us to pass to the limit as $\e\to0$, in all the terms of the equation, using Lebesgue's dominated convergence theorem, thus obtaining \eqref{vweq1}.

The same conclusion holds for equation \eqref{very weak eq x2}. Indeed, the boundedness of $v^\e$  and its convergence \eqref{a.e. v}, together with the estimates in \eqref{eq lemma u}, allow us to pass to the limit in all terms of \eqref{very weak eq x2}, using Lebesgue's dominated convergence theorem again, thus obtaining \eqref{vweq2}. The proof of Theorem \ref{thr 1} is now completed. \end{proof}
\section{Linear stability analysis}\label{sect stability}
In this section, we investigate the linear stability of  spatially homogeneous steady states of the macroscopic system \eqref{macro}--\eqref{BC}, with reaction and fast reaction functions given by \eqref{def react functions} and \eqref{def Q}, respectively. We shall also see the relationship between the linear stability of the coexistence steady state at the mesoscopic and macroscopic scale, as $\e\to0$.

Let $\psi$ and $\phi$ be conversion rates satisfying assumption (H1). We introduce the following few notations for later use,
\begin{equation*}
\psi_1=\psi(1),\qquad\phi_1=\phi(1)\,,
\end{equation*}
and the parameter providing a criterion for the linear stability (see \textit{Theorem~\ref{thm:stability macro}}
 and \textit{Proposition~\ref{prop uniq equi}}),
\begin{equation}\label{def alpha beta}
\alpha\coloneqq\f{\psi_1}{\phi_1}\,\f{a}{b}>0\,.
\end{equation}

The pair $(\bar u,\bar v)\in\R^2_+$ is a spatially homogeneous steady state of the macroscopic system if and only if $\bar u=\bar u_a+\bar u_b$ and the triplet $(\bar u_a,\bar u_b,\bar v)$ satisfy the nonlinear system
\begin{equation}\label{system steady states}
  f_a(\bar u_a)+f_b(\bar u_b,\bar v)=f_v(\bar u_b,\bar v)
  =Q(\bar u_a,\bar u_b,\bar v)=0\,.
\end{equation}

{\bf Extinction of $u$.} From $Q(\bar u_a,\bar u_b,\bar v)=0$ and the strict positivity of $\phi$ and $\psi$, we see that $\bar u_{a}=0$ if and only if $\bar u_{b}=0$: no extinction of a single subpopulation of the species $u$ is admitted. Thus, for $\bar u_{a}=\bar u_{b}=0$, we obtain the trivial and semi-trivial steady states
\begin{equation}\label{def trivial equi}
(\bar u_1, \bar v_1)=(0,0)\qquad \text{and} \qquad (\bar u_2, \bar v_2)=(0,b)\,,
\end{equation}
corresponding to the total extinction of the two species in the ecosystem and to a partial extinction, respectively.
\medskip

{\bf Survival of $u$ and extinction of $v$.} The other steady states
with $\bar u_{a}\neq0$ and $\bar u_{b}\neq0$ are of main interest. The
first case is with $\bar v=0$. Denoting $\bar u_{a}=\lambda\,a$ and
$\bar u_{b}=\sigma\,b$, for $\lambda,\sigma>0$,
system~\eqref{system steady states} reduces to
\begin{equation}\label{eq:quadratic and non linear v=0}
  \eta_aa\,\lambda(1-\lambda)+\eta_bb\,\sigma(1-\sigma)=0\,,\qquad
  \f{\lambda\psi(\lambda)}{\sigma\phi(\sigma)}=\f ba\,.
\end{equation}
Such a semi-trivial state always exists but the uniqueness is
non-trivial. Indeed, the second equation in \eqref{eq:quadratic and
  non linear v=0} can be written equivalently as
\begin{equation}\label{eq non linear sigma lambda}
  \f{\sigma\phi(\sigma)}{\phi_1}=\alpha\,\f{\lambda\psi(\lambda)}{\psi_1}\,.
\end{equation}
Due to assumption (H1), the functions
$\Lambda(\lambda):=\lambda\psi(\lambda)/\psi_1$ and
$\Sigma(\sigma):=\sigma\phi(\sigma)/\phi_1$ are strictly increasing
functions from 0 to $+\infty$. Hence, for every $\lambda>0$ there
exists a unique $\sigma(\lambda)>0$ solving \eqref{eq non linear sigma
  lambda} and given by
\begin{equation}\label{def sigma as inverse function}
  \sigma(\lambda)=\Sigma^{-1}(\alpha\Lambda(\lambda))\,.
\end{equation}
Plugging \eqref{def sigma as inverse function} into the left hand side equation in \eqref{eq:quadratic and non linear v=0}, the stationary states correspond to the zeros of the function $F$ below
\begin{equation}\label{def:F(lambda)}
F(\lambda):=\eta_aa\,\lambda(1-\lambda)+\eta_bb\,\sigma(\lambda)(1-\sigma(\lambda))\,.
\end{equation}
Furthermore, by the competition structure, it follows that $F$ is positive for small enough $\lambda$ and $F(\lambda)\to-\infty$ as $\lambda\to+\infty$. Thus, the macroscopic system \eqref{def react functions}--\eqref{NLsystem} admits at least one semi-trivial equilibrium
\begin{equation}\label{family equi v=0}
(\bar u_3,\bar v_3)=(a\lambda+b\sigma,0)\,,
\end{equation}
solution of system \eqref{eq:quadratic and non linear v=0}, with $\sigma=\sigma(\lambda)$ uniquely determined by \eqref{def sigma as inverse function}. Moreover, if the equilibrium is unique, $F$ is decreasing around the corresponding $\lambda$, i.e. $F'(\lambda)<~0$.

In general it is possible to have several semi-trivial states of type \eqref{family equi v=0}. As an example, take
\begin{equation}
  \label{eq:non-uniqueness-steady-state}
  a=b=1,\quad
  \eta_a = 0.2,\quad
  \eta_b = 1,\quad
  \phi \equiv 1,\quad
  \psi(x) =
  \begin{cases}
    0.1 & \text{if } x \le 1.6, \\
    0.3 & \text{otherwise}.
  \end{cases}
\end{equation}
The corresponding \(F(\lambda)\) is shown in Figure~\ref{fig:non-uniqueness-example}, from where we see that there exist three semi-trivial states.
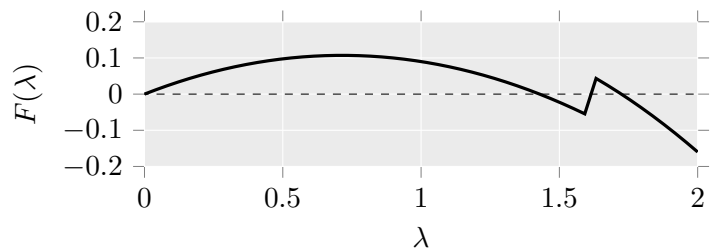
\begin{figure}[H]
  \centering
  \begin{tikzpicture}
    \begin{axis}[
      xlabel={\(\lambda\)},
      xmin=0.0, xmax=2.0,
      ylabel={\(F(\lambda)\)},
      ymin=-0.2, ymax=0.2,
      height=3.5cm,width=0.7\textwidth
      ]
      \addplot [very thick]
      table {%
        0 0
        0.0408163265306122 0.0118950437317784
        0.0816326530612245 0.0230903790087464
        0.122448979591837 0.0335860058309038
        0.163265306122449 0.0433819241982507
        0.204081632653061 0.0524781341107872
        0.244897959183673 0.0608746355685131
        0.285714285714286 0.0685714285714286
        0.326530612244898 0.0755685131195335
        0.36734693877551 0.081865889212828
        0.408163265306122 0.0874635568513119
        0.448979591836735 0.0923615160349854
        0.489795918367347 0.0965597667638484
        0.530612244897959 0.100058309037901
        0.571428571428571 0.102857142857143
        0.612244897959184 0.104956268221574
        0.653061224489796 0.106355685131195
        0.693877551020408 0.107055393586006
        0.73469387755102 0.107055393586006
        0.775510204081633 0.106355685131195
        0.816326530612245 0.104956268221574
        0.857142857142857 0.102857142857143
        0.897959183673469 0.100058309037901
        0.938775510204082 0.0965597667638484
        0.979591836734694 0.0923615160349854
        1.02040816326531 0.0874635568513119
        1.06122448979592 0.081865889212828
        1.10204081632653 0.0755685131195336
        1.14285714285714 0.0685714285714286
        1.18367346938776 0.0608746355685131
        1.22448979591837 0.0524781341107872
        1.26530612244898 0.0433819241982508
        1.30612244897959 0.0335860058309038
        1.3469387755102 0.0230903790087464
        1.38775510204082 0.0118950437317784
        1.42857142857143 5.55111512312578e-17
        1.46938775510204 -0.0125947521865889
        1.51020408163265 -0.0258892128279883
        1.55102040816327 -0.0398833819241982
        1.59183673469388 -0.0545772594752186
        1.63265306122449 0.0433152852977927
        1.6734693877551 0.0245897542690546
        1.71428571428571 0.0048979591836735
        1.75510204081633 -0.0157600999583506
        1.79591836734694 -0.0373844231570178
        1.83673469387755 -0.0599750104123281
        1.87755102040816 -0.0835318617242815
        1.91836734693878 -0.108054977092878
        1.95918367346939 -0.133544356518117
        2 -0.16
      };
      \addplot [dashed,thin] {
        0
      };
    \end{axis}
  \end{tikzpicture}
  \caption{Reaction term \(F(\lambda)\) for the example
    \eqref{eq:non-uniqueness-steady-state}.}
  \label{fig:non-uniqueness-example}
\end{figure}

We will discuss the uniqueness issue in \textit{Proposition} \ref{prop uniq equi}, where a sufficient condition for uniqueness of \eqref{family equi v=0} is given, and \textit{Proposition} \ref{prop uniq equi linear}, where we  exhibit a family of conversion rates functions $\phi,\psi$ for which uniqueness of \eqref{family equi v=0} holds true.
\medskip

{\bf Coexistence of $u$ and $v$.} Finally, if $\bar u_{a}\neq0$, $\bar u_{b}\neq0$, $\bar v\neq0$, from $f_v(\bar u_b,\bar v)=0$ we get $\bar u_b+\bar v=b$ and thus $\bar u_a=a$. Then, from $Q(\bar u_a,\bar u_b,\bar v)=0$ and the definition of $\alpha$ it follows that $\bar u_b=b\alpha$. Therefore, system \eqref{system steady states} has a unique totally nontrivial solution given by
\begin{equation}\label{def coex equi}
(\bar u_4,\bar v_4)=(a+b\alpha,b(1-\alpha))\,,
\end{equation}
provided that $\alpha<1$.
\medskip

We shall see in the following subsection (see \textit{Theorem} \ref{thm:stability macro}) that the stationary states \eqref{def trivial equi} are unstable, so that the total extinction of the species $u$ never occurs. The species $u$ always survives  and its coexistence with the species $v$ is conditioned by the switching strategy that the subpopulations $u_a$ and $u_b$ adopt when both resources run out, quantified through the parameter $\alpha$. Indeed, the coexistence occurs if the switch from the state $u_b$ to the state $u_a$ is faster than the opposite switch, i.e. $\alpha<1$. On the other hand, $v$ goes extinct only if $\alpha>1$.

The relationship between the linear stability of the mesoscopic and macroscopic coexistence steady states, as $\e\to0$, is seen in \textit{Subsection \ref{stability meso}}.
\subsection{Linear stability analysis for the cross-diffusion system}\label{subsect stability}
Let us consider the partial starvation measures
\[
\lambda=\f{\bar u_a}a\ge0\,,\qquad \sigma=\f{\bar u_b}b\ge0\,,\qquad \delta=\f{\bar v}b\in\{0,1-\sigma\}\,,\qquad
\]
so that each of the above steady states can be identified with the triplet $(\lambda,\sigma,\delta)$ and written as
\begin{equation}\label{def P}
\bar{P}=(\bar u,\bar v)=\big(\lambda a+\sigma b,\delta b\big)\,.
\end{equation}

Linearizing around $\bar{P}$ the ODEs system associated to \eqref{def react functions}--\eqref{NLsystem}, in the sense of small perturbation $\tau$, $|\tau|\ll1$, i.e.
\begin{equation}\label{def u v perturbated}
\begin{split}
&u_a=\bar u_a+\tau\,\tilde{u}_a\quad\text{and}\quad u_b=\bar u_b+\tau\,\tilde{u}_b \\
&u=u_a+u_b=(\bar u_a+\bar u_b)+\tau(\tilde{u}_a+\tilde{u}_b)=\bar u+\tau\,\tilde{u},\\
& v=\bar v+\tau\tilde{v},
\end{split}
\end{equation}
we obtain
\begin{equation}\label{ODE linearized}
\begin{cases}
\dot{\tilde{u}}=\eta_a(1-2\lambda)\tilde{u}_a+\eta_b(1-2\sigma-\delta)\tilde{u}_b-\eta_b\sigma\,\tilde{v}+o(1),\\
\dot{\tilde{v}}=-\eta_v\delta\,\tilde{u}_b+\eta_v(1-\sigma-2\delta)\tilde{v}+o(1).
\end{cases}
\end{equation}
Moreover, from the linearization of $Q(u_a,u_b,v)$ around $(\bar u_a,\bar u_b,\bar v),$ we have
\begin{equation}\label{lin Q}
\pa_1 \bar Q\,\tilde{u}_a+\pa_2 \bar Q\,\tilde{u}_b+\pa_3 \bar Q\,\tilde v+o(1)=0,
\end{equation}
where $\pa_j \bar Q=\pa_j Q(\bar u_a,\bar u_b,\bar v)$ and
\begin{equation}\label{def partial Q}
\begin{split}
\pa_1 \bar Q&=-\psi(\lambda)-\lambda\psi'(\lambda)=:-\beta(\lambda)<0\,,\\
\pa_2 \bar Q&=\phi(\sigma+\delta)+\sigma\phi'(\sigma+\delta)=:\gamma(\sigma,\delta)>0\,,\\
\pa_3 \bar Q&=\sigma\phi'(\sigma+\delta)=:\theta(\sigma,\delta)>0\,.
\end{split}
\end{equation}
Using $\tilde u=\tilde{u}_a+\tilde{u}_b$, from \eqref{lin Q} we obtain $\tilde u_a$ and $\tilde u_b$ in terms of $\tilde u$ and $\tilde v$ as follows
\begin{equation}\label{eq: ua tilde ub tilde}
\tilde{u}_a=\f1r\,\gamma(\sigma,\delta)\, \tilde{u}+\f1r\,\theta(\sigma,\delta)\,\tilde{v}+o(1)\,,\qquad
\tilde{u}_b=\f1r\,\beta(\lambda)\,\tilde{u}-\f1r\,\theta(\sigma,\delta)\,\tilde{v}+o(1)\,,
\end{equation}
where $r= r(\lambda,\sigma,\delta)\coloneqq\pa_2\bar Q-\pa_1\bar Q=\beta(\lambda)+\gamma(\sigma,\delta)>0$. Thus, system \eqref{ODE linearized} becomes
\[
\dot{\tilde{w}}=\bar M\,\tilde{w} +o(1),\qquad\qquad\tilde{w}\coloneqq
\begin{pmatrix}
\,\tilde{u}\,\\\,\tilde{v}\,
\end{pmatrix},
\]
and the matrix $\bar M=M(\bar P)$ has the following entries
\begin{equation}\label{def matrix M}
\begin{split}
M_{11}(\bar P)&=\f{\eta_a}r(1-2\lambda)\gamma(\sigma,\delta)+\f{\eta_b}r(1-2\sigma-\delta)\beta(\lambda)\,,\\
M_{12}(\bar P)&=\f{\eta_a}r(1-2\lambda)\theta(\sigma,\delta)-\f{\eta_b}r(1-2\sigma-\delta)\theta(\sigma,\delta)-\eta_b\sigma\,,\\
M_{21}(\bar P)&=-\f{\eta_v}r\delta\,\beta(\lambda)\,,\\
M_{22}(\bar P)&=\f{\eta_v}r\delta\,\theta(\sigma,\delta)+\eta_v(1-\sigma-2\delta)\,.
\end{split}
\end{equation}

Next, for $u_a$ and $u_b$ as in \eqref{def u v perturbated}, using
\eqref{eq: ua tilde ub tilde} again, the linearization of the
cross-diffusion operator in \eqref{macro} reads as
\[
  \Delta\big(d_au_a+d_bu_b\big)
  =\tau \Big(d_a\f{\gamma(\sigma,\delta)}r+d_b\f{\beta(\lambda)}r\Big)\Delta \tilde u+\tau(d_a-d_b)\f{\theta(\sigma,\delta)}r\Delta \tilde v+o(\tau),
\]
and the linearized cross-diffusion macroscopic system writes
\begin{equation}\label{linearized cross-diff system}
  \partial_{t}{\tilde{w}}
  =\bar J\Delta\tilde w+\bar M\tilde{w}+o(1)\,,\qquad \nabla(\tilde{w}+o(1))\cdot\sigma=0,
\end{equation}
with
\begin{equation*}\label{def J}
  \bar J\coloneqq
  \begin{bmatrix}
    d_a\f{\gamma(\sigma,\delta)}r+d_b\f{\beta(\lambda)}r\quad&\quad (d_a-d_b)\f{\theta(\sigma,\delta)}r\\[1.5ex]
    0& d_v
  \end{bmatrix}.
\end{equation*}
The homogeneous (up to a $o(1)$ term) Neumann boundary conditions for $\tilde w$ in \eqref{linearized cross-diff system} follow by the no flux boundary condition \eqref{BC} and \eqref{eq: ua tilde ub tilde}.

Neglecting the $o(1)$ terms, the stability of the linearized system \eqref{linearized cross-diff system} can be analysed decomposing $\tilde{w}(t,x)$  as
\begin{equation*}
\tilde{w}(t,x)=\sum_{n\in\N} \tilde{w}_n(t)\, e_n(x)\,,
\end{equation*}
where \((e_n)_{n\in\N}\) is the orthogonal eigenbasis  of \(-\Delta\) on $\Omega$ with Neumann boundary
conditions. Denoting $0= \lambda_0 < \lambda_1 \le \dots \le \lambda_n \le \dotsb$ the corresponding eigenvalues, the projection coefficients $\tilde{w}_n(t)$ evolve independently according the equations
\begin{equation*}
  \partial_t \tilde{w}_n(t)= (-\lambda_n \bar{J} + \bar M) \tilde{w}_n(t)\,,\qquad n\in\N\,.
\end{equation*}
Thus, for the stability analysis it suffices to consider the stability of the matrix \(N_n\coloneqq-\lambda_n \bar{J} + \bar M\), i.e.
\begin{equation}\label{def N}
  N_n =
  \begin{bmatrix}
    -\f{1}r\big(d_a\,\gamma+d_b\beta\big)\lambda_n+M_{11} \quad
    &\quad -\f1r(d_a-d_b)\theta\,\lambda_n+M_{12}\\[1.5ex]
    M_{21} & -d_v\lambda_n+M_{22}
  \end{bmatrix},
\end{equation}
with $M_{ij}= M_{ij}(\bar P)$ defined in \eqref{def matrix M}.
\begin{thm}\label{thm:stability macro}
Let $\psi$ and $\phi$ be conversion rates satisfying assumption (H1) and $\alpha>0$ defined as in \eqref{def alpha beta}. Then, the following holds true.
\begin{itemize}
\item[(i)] The trivial and semi-trivial steady states $(\bar u_1,\bar v_1)=(0,0)$ and $(\bar u_2,\bar v_2)=(0,b)$ are linearly unstable.
\item[(ii)] The family of semi-trivial steady states $(\bar u_3,\bar v_3)=(a\lambda+b\sigma,0)$ satisfies
\begin{equation}\label{upperbounds0}
\sigma=\lambda=1\,,\qquad\text{if } \alpha=1\,,
\end{equation}
\begin{equation}\label{upperbounds1}
0<\sigma<1<\lambda<\f12+\f12\sqrt{1+\f{b\eta_b}{a\eta_a}}\,,\qquad\text{if } \alpha<1\,,
\end{equation}
and the swapped relation
\begin{equation}\label{upperbounds2}
0<\lambda<1<\sigma<\f12+\f12\sqrt{1+\f{a\eta_a}{b\eta_b}}\,,\qquad\text{if } \alpha>1\,.
\end{equation}
Furthermore, they are linearly unstable if $\alpha\le1$, and if $\alpha>1$, they are linearly stable if and only if the function $F$ in \eqref{def:F(lambda)} is strictly decreasing around $\lambda$, i.e. $F'(\lambda)<0$.
\item[(iii)]  If $\alpha<1$, there exists a unique strictly positive steady state given by $(\bar u_4,\bar v_4)=(a+b\alpha,b(1-\alpha))$ and it is linearly stable.
\end{itemize}
\end{thm}
\begin{proof}
\textit{(i)} From \eqref{def matrix M} and \eqref{def partial Q}, we have
\[
M(0,0)=\diag\Big\{\f{\eta_a\phi(0)+\eta_b\psi(0)}{\phi(0)+\psi(0)},\,\eta_v\Big\}\qquad\text{and}\qquad
M(0,b)=
\begin{bmatrix}
\f{\eta_a\phi_1}{\phi_1+\psi(0)} \,&\,0\\
-\f{\eta_v\psi(0)}{\phi_1+\psi(0)}\,& \,-\eta_v
\end{bmatrix},
\]
implying that the steady states $(0,0)$ and $(0,b)$ are linearly unstable, both for the macroscopic system and for the associated diffusion-less one, because of the zero eigenvalue of the Laplacian.
\medskip


\textit{(ii)} In order to proceed with the investigation of the family of steady states $(\bar u_3,\bar v_3)=(a\lambda+b\sigma,0)$, let us observe that from the first equation in \eqref{eq:quadratic and non linear v=0}, we have
\begin{equation}\label{rel lambda sigma}
(1-\lambda)(1-\sigma)<0\quad \text{or}\quad \lambda=\sigma=1\,.
\end{equation}
Thus, according to the value of $\alpha$, we get from \eqref{eq non linear sigma lambda}: if $\alpha>1$, then $\lambda\in(0,1)$ and $\sigma>1$, i.e. $\bar u_a<a$ and $\bar u_b>b$; if $\alpha<1$, then $\lambda>1$ and $\sigma\in(0,1)$, i.e. $\bar u_a>a$ and $\bar u_b<b$; if $\alpha=1$, then $\lambda=\sigma=1$ giving the optimal selection case $\bar u_{a}=a,\bar u_{b}=b$.

Next, let us rewrite the left equation in \eqref{eq:quadratic and non linear v=0} as
\begin{equation}\label{def K}
\sigma(1-\sigma)=\f{\eta_aa}{\eta_bb}\lambda(\lambda-1)=:K(\lambda)\,.
\end{equation}
If $\alpha>1$, as $\lambda\in(0,1)$, it follows that $K(\f12)\le K(\lambda)<0$ and $\sigma$ is upper bounded by the positive root of the above equation with $\lambda=\f12$. Hence,  \eqref{upperbounds2} follows. If $\alpha<1$, swapping the role between $\lambda$ and $\sigma$, we obtain  \eqref{upperbounds1}.

Furthermore, the entries \eqref{def matrix M} of the matrix $M(\bar P)=M(a\lambda+b\sigma,0)$ are now
\[
\begin{split}
M_{11}(\bar P)&=\eta_a(1-2\lambda)\f{\gamma}{r}+\eta_b(1-2\sigma)\f{\beta}{r}\,,\\
M_{12}(\bar P)&=\big(\eta_a(1-2\lambda)-\eta_b(1-2\sigma)\big)\f{\theta}{r}-\eta_b\,\sigma\,,\\
M_{21}(\bar P)&=0\,,\\
M_{22}(\bar P)&=\eta_v(1-\sigma)\,.
\end{split}
\]
As $M_{21}=0$, the steady state is linearly stable for the
diffusionless macroscopic system if and only if $M_{11}<0$ and
$M_{22}<0$. Hence, $\sigma>1$ is a necessary condition for the linear
stability, and it holds only if $\alpha>1$.

In the case $\alpha=1$, giving the optimal selection case
$\lambda=\sigma=1$, $M(a + b, 0)$ has a zero eigenvalue, so that the
equilibrium is a non hyperbolic equilibrium. The contribution of the
cross-diffusion term does not change the nature of the equilibrium
because of the zero eigenvalue of the Laplacian.

Let $\alpha>1$. The steady states under consideration satisfy
$Q(\lambda a, \sigma(\lambda)b,0)=0$, where $\sigma(\lambda)$ is
defined in \eqref{def sigma as inverse function}. Taking the
derivative with respect to $\lambda$ and using \eqref{def partial Q},
we obtain
\[
  a\pa_1Q(\lambda a,\sigma(\lambda)b,0)
  +b\,\sigma'(\lambda)\pa_2Q(\lambda a, \sigma(\lambda)b,0)
  =-\beta(\lambda)a+\gamma(\sigma(\lambda),0)b\,\sigma'(\lambda)=0\,.
\]
Thus
\[
  \sigma'(\lambda)=\f ab\f{\beta(\lambda)}{\gamma(\sigma(\lambda),0)}\,.
\]
Plugging $\sigma'(\lambda)$ into the derivative of $F$
\begin{equation}\label{F'(lambda)}
  F'(\lambda)=\eta_a a (1-2\lambda) + \eta_b b\,\sigma'(\lambda) (1-2\sigma(\lambda))\,,
\end{equation}
we now find
\[
  F'(\lambda)= \eta_a a (1-2\lambda) +
  \eta_ba\f{\beta(\lambda)}{\gamma(\sigma(\lambda),0)}(1-2\sigma(\lambda))
  = \frac{a}{\gamma(\sigma(\lambda),0)}\, r M_{11}(\bar P).
\]
Hence, $M_{11}$ is negative if and only if $F'(\lambda)$ is negative, which implies $(ii)$ for the diffusionless macroscopic system and for the cross-diffusion one.
\medskip

\textit{(iii)} Let $\alpha<1$. Since now $(\lambda,\sigma,\delta)=(1,\alpha,1-\alpha)$, from \eqref{def matrix M}, we have
\begin{equation}\label{def M coex case}
M(\bar u_4,\bar v_4)=-\f{1}{r}
\begin{bmatrix}
\eta_a\gamma+\eta_b\alpha\beta\quad&\quad \eta_a\theta+\eta_b\alpha(r-\theta)\\
\eta_v(1-\alpha) \beta\quad&\quad\eta_v (1-\alpha)(r-\theta)
\end{bmatrix}.
\end{equation}
As $r-\theta>0$, it holds
\begin{equation}\label{tr M}
\text{tr} M<0\,.
\end{equation}
By $r=\beta+\gamma$ and $\gamma-\theta=\phi_1$, we have
\begin{equation}\label{det M}
  \begin{split}
    \det M &=
    \f {\eta_v(1-\alpha)} {r^2}\big[(\eta_a\gamma+\eta_b\alpha\beta)(r-\theta)-\eta_a\theta\beta-\eta_b\alpha\beta(r-\theta)\big]\\
    &= \f {\eta_a\eta_v(1-\alpha)} {r^2}\big[\gamma(r-\theta)-\theta\beta\big]
    =\f{\eta_a\eta_v(1-\alpha)}r\,\phi_1>0,
  \end{split}
\end{equation}
i.e. the equilibrium $(\bar u_4,\bar v_4)$ is stable for the diffusionless macroscopic system.

The expression form \eqref{def M coex case} for $M$ implies for $N_n$, \(n\in\N\), by \eqref{def N}, that
\[
  \tr N_n<0\,,
\]
and
\[
  \det{N_n}=A\lambda_n^2+B\lambda_n+C\,,
\]
with
\begin{align}
  A &\coloneqq d_v\f{d_a\gamma+d_b\beta}r>0,\notag\\
  B &\coloneqq\f{ (d_a-d_b)\theta}r M_{21}-\f{d_a\gamma+d_b\beta}rM_{22}-d_vM_{11},\label{def A2}\\
  C &\coloneqq \det{M}>0.\notag
\end{align}
Furthermore, using the definition of $r$ and the strict negativity of
all the entries of $M(\bar u_4,\bar v_4)$, we find for $B$
in~\eqref{def A2}
\begin{align*}
  B&=-(d_a-d_b)\f{\eta_v\theta\beta(1-\alpha)}{r^2}+(d_a\gamma+d_b\beta)\f{\eta_v(r-\theta)(1-\alpha)}{r^2}-d_vM_{11}\\
   &=\f{\eta_v(1-\alpha)}{r^2}\Big(-d_a\theta\beta+d_ar\gamma-d_a\theta\gamma+d_br\beta\Big)-d_vM_{11}\\
   &=\f{\eta_v(1-\alpha)}{r}(d_a\phi_1+d_b\beta)-d_vM_{11}>0,
\end{align*}
which implies that $\det N_n>0,$ for all $n\in\N.$ Therefore, the
equilibrium $(\bar u_4,\bar v_4)$ remains linearly stable by adding the
cross-diffusion terms.
\end{proof}
%
\subsection{Uniqueness of semi-trivial states with extinction of $v$}
One possibility to ensure uniqueness of the steady state $(\bar u_3,\bar v_3)=(a\lambda+b\sigma,0)$ is to impose, in the case $\alpha>1$, that the net flux of the individuals of the species $u$ goes from the state $u_b$ to the state $u_a$, when the population \(u_b\) reached the capacity of its resource and the population \(u_a\) has only halved the capacity of its resource. When $\alpha<1$, the opposite switching mechanism has to be imposed. A precise version is the following.
\begin{prop}\label{prop uniq equi}
Consider $\Lambda(\lambda)=\lambda\psi(\lambda)/\psi_1$ and $\Sigma(\sigma)=\sigma\phi(\sigma)/\phi_1$, with $\phi$, $\psi$ satisfying assumption (H1). Assume that
\begin{equation}\label{uniq alpha>1}
\alpha\Lambda(1/2) \le1\,,\qquad\text{if}\qquad\alpha>1\,,
\end{equation}
and
\begin{equation}\label{uniq alpha<1}
\alpha^{-1}\Sigma(1/2) \le1\,,\qquad\text{if}\qquad\alpha<1\,.
\end{equation}
Then, there exists a unique solution of \eqref{eq:quadratic and non linear v=0}. Furthermore, the corresponding steady state \eqref{family equi v=0} is linearly stable if $\alpha>1$, and unstable if $\alpha<1$.
\end{prop}
\begin{proof}
Let $\alpha>1$. For the proof recall the function \(\lambda \mapsto \sigma(\lambda)\) from \eqref{def sigma as inverse function}. Then, $\sigma(0)=0$, while the increasing behaviour of $\Lambda$ and $\Sigma$ together with condition \eqref{uniq alpha>1} imply that, for $\lambda\in(0,1/2]$,
\[
\sigma(\lambda)\le\Sigma^{-1}(\alpha\Lambda(1/2))\le\Sigma^{-1}(1)=1\,.
\]
Hence, for \(\lambda \in (0,1/2]\), the function \(F\) from \eqref{def:F(lambda)} is strictly positive.

Now, let \(\bar{\lambda}\) be the smallest zero of \(F\), so that $(a\bar\lambda+b\sigma(\bar\lambda),0)$ is one of the steady states under consideration. By the above argument \(\bar{\lambda} > 1/2\), and by \textit{Theorem}~\ref{thm:stability macro}, \(\alpha>1\) implies that \(\sigma(\bar{\lambda}) > 1\).  Therefore, the monotonicity of \(\lambda \mapsto \sigma(\lambda)\) again implies that \(\sigma(\lambda) > 1\), for any \(\lambda \ge \bar{\lambda}\).

Finally, we find from \eqref{F'(lambda)} that \(F'(\lambda) < 0\), for all \(\lambda \ge \bar{\lambda}\). Hence there exists a unique stationary state and the claimed stability follows from \textit{Theorem}~\ref{thm:stability macro}.

The case $\alpha<1$ follows changing the role between the variables $\lambda$ and $\sigma$ and between the functions $\Lambda$ and $\Sigma$, i.e. defining $\lambda(\sigma):=\Lambda^{-1}(\alpha^{-1}\Sigma(\sigma))$ and analyzing the behaviour of $G(\sigma):=\eta_a a\lambda(\sigma)(1-\lambda(\sigma))+\eta_b b\sigma(1-\sigma)$, instead of $F(\lambda)$. The claimed instability follows again by \textit{Theorem}~\ref{thm:stability macro}.
\end{proof}

Conditions \eqref{uniq alpha>1} and \eqref{uniq alpha<1} can be
rephrased in terms of the ratio $\f ba$, respectively as
\[
  \f{\f12\psi(\f12)}{\phi_1}\le \f ba<\f{\psi_1}{\phi_1}\qquad\text{and}\qquad
  \f{\psi_1}{\phi_1}<\f ba\le \f{\psi_1}{\f12\phi(\f12)}\,.
\]
They are not necessary conditions. Indeed, we provide below a family
of conversion rates $\psi,\phi,$ for which the uniqueness of the
stationary states \eqref{family equi v=0} holds true, whatever is
$\f ba$. For that family of conversion rates, some numerical test are
shown in \textit{Section}~\ref{sect.simulation}.

Since the population densities $u_a$ and $u_b$ are of the same
species, it is natural to expect that the conversion dynamics from
$u_a$ to $u_b$ is similar to that from $u_b$ to $u_a$. So, in order to
be consistent with the modelling considerations in
\textit{Subsection}~\ref{sect micro derivation}, (see
\eqref{def:phi-psi}), we choose
\begin{equation}\label{def psi}
  \psi(x)=\omega_1\phi(\omega_2x),\qquad\omega_1>0\,,\omega_2\ge0\,,
\end{equation}
and we prove the following.
\begin{prop}\label{prop uniq equi linear}
Consider  $\psi $ as in \eqref{def psi} and
\begin{equation}\label{def phi}
\phi(x)=\theta_1x+\theta_2,\qquad\theta_1\ge0\,,\theta_2>0\,,
\end{equation}
Then there exists a unique stationary state $(\bar u_3,\bar v_3)=(a\lambda+b\sigma,0)$. It is linearly stable if $\f ba<\omega_1\phi(\omega_2)/\phi_1$, and unstable otherwise.
\end{prop}
\begin{proof}
Let $\sigma(\lambda)$ be as in \eqref{def sigma as inverse function}. As observed previously, the stationary states \eqref{family equi v=0} corresponds to the zeros of the function $F(\lambda)$ in \eqref{def:F(lambda)}. Taking the second derivative of $F$, gives
\begin{equation}\label{second der F}
F''(\lambda)=b\,\eta_b\big[\sigma''(\lambda)-2(\sigma'(\lambda))^2-2\sigma(\lambda)\sigma''(\lambda)\big]-2a\,\eta_a\,.
\end{equation}
By \eqref{def phi} and  \eqref{def psi}, we have
\[
\f{\sigma\phi(\sigma)}{\phi_1}=\bar\theta\sigma^2+(1-\bar\theta)\sigma\,,\qquad \bar\theta=\f{\theta_1}{\theta_1+\theta_2}\,,
\]
and
\[
\f{\lambda\psi(\lambda)}{\psi_1}=\bar\omega\lambda^2+(1-\bar\omega)\lambda\,,\qquad \bar\omega=\f{\omega_2\theta_1}{\omega_2\theta_1+\theta_2}\,.
\]
Hence, equation \eqref{eq non linear sigma lambda} reads as
\begin{equation}\label{eq for sigma(lambda)}
\bar\theta\sigma^2(\lambda)+(1-\bar\theta)\sigma(\lambda)=\alpha[\bar\omega\lambda^2+(1-\bar\omega)\lambda]=:W(\lambda)\,,
\end{equation}
and
\[
\sigma(\lambda)=\f{\bar\theta-1}{2\bar\theta}+\f1{2\bar\theta}[(\bar\theta-1)^2+4\bar\theta\,W(\lambda)]^{\f12}\,.
\]
Furthermore, deriving twice \eqref{eq for sigma(lambda)} with respect to $\lambda$, we obtain the identity
\[
2(\sigma'(\lambda))^2+2\sigma(\lambda)\sigma''(\lambda)=2\alpha\f{\bar\omega}{\bar\theta}+(1-\f1{\bar\theta})\sigma''(\lambda)\,.
\]
Plugging the latter into \eqref{second der F}, we end up with
\[
F''(\lambda)=\f{b\,\eta_b}{\bar\theta}\sigma''(\lambda)-(2\alpha\f{\bar\omega}{\bar\theta}b\,\eta_b+2a\,\eta_a)\,.
\]
Finally, observing that $W'^2-2W\,W''=\alpha^2(1-\bar\omega)^2$, we compute
\[
  \begin{aligned}
  \sigma''(\lambda)
  &=\left(\frac{W'(\lambda)}{[(\bar\theta-1)^2+4\bar\theta\,W(\lambda)]^{\f12}}\right)'
  =\f{W''[(\bar\theta-1)^2+4\bar\theta\,W]-2\bar \theta W'^2 }{[(\bar\theta-1)^2+4\bar\theta\,W]^{\f32}} \\ \\
  &=\f{2\alpha\bar\omega(\bar\theta-1)^2-2\bar\theta(W'^2-2W\,W'')}{[(\bar\theta-1)^2+4\bar\theta\,W]^{\f32}}
  =2\alpha\f{\bar\omega(1-\bar\theta)^2-\alpha\bar\theta(1-\bar\omega)^2}{[(\bar\theta-1)^2+4\bar\theta\,W(\lambda)]^{\f32}}\,.
  \end{aligned}
\]
If $\bar\omega(1-\bar\theta)^2-\alpha\bar\theta(1-\bar\omega)^2\le0$, the function $F$ is strictly concave and therefore has a unique zero. If $\bar\omega(1-\bar\theta)^2-\alpha\bar\theta(1-\bar\omega)^2>0$, then $\sigma''(\lambda)$ is a strictly positive decreasing function that converge to $0$ as $\lambda\to+\infty$, and consequently $F$ has at most one inflection point and a unique zero. Moreover, $F$ is decreasing around its unique zero. So that it gives a stable stationary point if $\alpha>1$.
\end{proof}
\subsection{Linear stability analysis for the mesoscopic system}\label{stability meso}
A triple $(\bar u_a^\e,\bar u_b^\e,\bar v^\e)$ is a homogeneous stationary solutions of the mesoscopic scale problem \eqref{meso system} if and only if
\[
f_a(\bar u_a^\e)+\frac1\var Q(\bar u_a^\e,\bar u_b^\e,\bar v^\e)=f_b(\bar u_b^\e,\bar v^\e)-\frac1\var Q(\bar u_a^\e,\bar u_b^\e,\bar v^\e)=f_v(\bar u_b^\e,\bar v^\e)=0.
\]
If $\bar v^\e=0$, then either $\bar u_a^\e=\bar u_b^\e=0$, which gives the totally trivial steady state corresponding to the trivial macroscopic one $(\bar u_1,\bar v_1)$, or $\bar u_a^\e\neq0$ and $\bar u_b^\e\neq0$. In the second case the triplet $(\bar u_a^\e,\bar u_b^\e,0)$ satisfies the system
\[
\left\{
\begin{split}
&\eta_a\bar u_a^\e(1-\dfrac{\bar u_a^\e}{a})+\frac1\e\big[\phi(\dfrac{\bar u_b^\e}{b})\, \bar u_b^\e -\psi(\dfrac{\bar u_a^\e}{a})\,\bar u_a^\e\big]=0,\\
&\eta_b\bar u_b^\e(1-\dfrac{\bar u_b^\e}{b})-\frac1\e\big[\phi(\dfrac{\bar u_b^\e}{b})\,\bar u_b^\e -\psi(\dfrac{\bar u_a^\e}{a})\,\bar u_a^\e\big]=0,
\end{split}
\right.
\]
it can be non unique, as in the macroscopic case, and it converges to a macroscopic equilibrium $(\bar u_3,\bar v_3)$, in the limit $\e\to0$.

If $\bar v^\e\neq0$, then from $f_v(u_b,v)=0$ we have $\bar u_b^\e+\bar v^\e=b$. Hence, for all $\e>0$, $f_b(\bar u_b^\e,\bar v^\e)=Q(\bar u_a^\e,\bar u_b^\e,\bar v^\e)=0$ and we obtain the two stationary states $(\bar u_a^\e,\bar u_b^\e,\bar v^\e)=(0,0,b)$ and
\begin{equation}\label{coex equi eps}
(\bar u_a^\e,\bar u_b^\e,\bar v^\e)=(a,b\alpha,b(1-\alpha))\,,
\end{equation}
provided that $\alpha<1$. These equilibria do not depend on $\e>0$, so that we shall drop the $\e$ exponent in the sequel. In the limit $\e\to0$, they correspond to the linearly unstable equilibrium $(\bar u_2,\bar v_2)$ and to the positive linearly stable equilibrium  $(\bar u_4,\bar v_4)$,  respectively.

Hereafter, we focus on the totally nontrivial spatially homogeneous steady \eqref{coex equi eps}, and we see that, for all $\e>0$, it is also stable for the mesoscopic system \eqref{meso system} and the corresponding ODEs system. Indeed, setting
\begin{equation*}\label{def ua ub v eps perturbated}
u^\e_a=\bar{u}_a+\tau\tilde{u}^\e_a\qquad u^\e_b=\bar{u}_b+\tau\tilde{u}^\e_b,\qquad v^\e=\bar{v}+\tau\tilde{v}^\e,\qquad|\tau|\ll1,
\end{equation*}
the linearization of \eqref{meso system} around $(\bar u_a,\bar u_b,\bar v)$ writes as
\begin{equation*}\label{meso system linearized}
\partial_t \tilde w^\e=\diag\{d_a, d_b, d_v\}\Delta \tilde w^\e+ M^\e\tilde w^\e+o(1),\qquad\qquad\tilde{w}^\e\coloneqq\big(
\tilde{u}^\e_a,\tilde{u}^\e_b,\tilde{v}^\e\big)^{\mathrm{T}},
\end{equation*}
with
\begin{equation*}\label{def matrix M eps}
M^\e\coloneqq
\begin{bmatrix}
-\eta_a+\f{1}\e\pa_1\bar Q \qquad&\qquad \f{1}\e\pa_2\bar Q\qquad&\qquad\f{1}\e\pa_3\bar Q\\[1.5ex]
-\f{1}\e\pa_1\bar Q & -\eta_b\alpha-\f{1}\e\pa_2\bar Q\qquad& -\eta_b\alpha-\f{1}\e\pa_3\bar Q\\[1.5ex]
0&-\eta_v(1-\alpha)&-\eta_v(1-\alpha)
\end{bmatrix}.
\end{equation*}
Again, we need to analyse the stability of the matrix $M^\e$ above and $N^\e_n$ below
\begin{equation*}
N^\e_n\coloneqq-\lambda_n\diag\{d_a, d_b, d_v\}+M^\e,
\end{equation*}
i.e.
\begin{equation*}
N^\e_n=
\begin{bmatrix}
-d_a\lambda_n-\eta_a+\f{1}\e\pa_1\bar Q \,&\, \f{1}\e\pa_2\bar Q\,&\,\f{1}\e\pa_3\bar Q\\[1.5ex]
-\f{1}\e\pa_1\bar Q \,&\, -d_b\lambda_n-\eta_b\alpha-\f{1}\e\pa_2\bar Q\,\,& -\eta_b\alpha-\f{1}\e\pa_3\bar Q\\[1.5ex]
0&-\eta_v(1-\alpha)&-d_v\lambda_n-\eta_v(1-\alpha)
\end{bmatrix}.
\end{equation*}
For that, we apply the Routh-Hurwitz criterion \cite{RH} and we obtain the result below, proved in Appendix \ref{Appendix A}.

\begin{prop}\label{Prop Routh M eps}
Under the assumption \(\alpha<1\), for all $\e>0$ and $\lambda_n\ge0$, the matrices $M^\e$ and $N^\e_n$ are stable, i.e. all their eigenvalues have negative real part.
\end{prop}

To complete the analysis, we shall see below how the previous linear stability property is preserved in the limit as $\e\to0$. Indeed, two eigenvalues of $N^\e_n$ converge to those of $N_n$ in \eqref{def N}, while the third one goes to $-\infty$.

Let us denote
\begin{equation*}
D^{\e}(\mu)\coloneqq N^\e_n-\mu I_3,
\end{equation*}
where $I_3$ stands for the $3\times3$ identity matrix. The goal of the
computations below is to compute $|D^\e|=\det D^\e(\mu)$, (see also \cite{IMN}).

First, adding the second row of $D^\e$ to the first one, we get
\[
|D^{\e}|=
\begin{vmatrix}
-(d_a\lambda_n+\eta_a+\mu)&-(d_b\lambda_n+\eta_b\alpha+\mu)&-\eta_b\alpha\\[1.5ex]
-\f{1}\e\pa_1\bar Q&-(d_b\lambda_n+\eta_b\alpha+\mu)-\f{1}\e\pa_2\bar Q&-\eta_b\alpha-\f{1}\e\pa_3\bar Q\\[1.5ex]
0&-\eta_v(1-\alpha)&-\big(d_v\lambda_n+\eta_v(1-\alpha)+\mu\big)
\end{vmatrix}.
\]
  Recalling from \eqref{def partial Q} that $r=\pa_2\bar Q-\pa_1\bar
  Q$ and $\pa_3\bar Q=\sigma\,\phi'(\sigma+\delta)$, we find for
  $\sigma=\alpha$ and $\delta=1-\alpha$ that
\[
  \pa_3\bar Q+(\pa_1\bar Q-\pa_2\bar Q)\f{\alpha\phi'_1}r = 0.
\]
Adding to the third column the difference between the first and the second column, both multiplied by $\f{\alpha\phi'_1}r$, we thus obtain
\[
  |D^{\e}|=
  \begin{vmatrix}
    -(d_a\lambda_n+\eta_a+\mu)\quad&\quad-(d_b\lambda_n+\eta_b\alpha+\mu)\quad&\quad N_{12}\,\\[1.5ex]
    -\f{1}\e\pa_1\bar Q\quad&\quad-(d_b\lambda_n+\eta_b\alpha+\mu)-\f{1}\e\pa_2\bar Q\quad&\quad d_{23}\,\\[1.5ex]
    0\quad&\quad-\eta_v(1-\alpha)\quad& N_{22}-\mu\quad
  \end{vmatrix},
\]
with $N_{ij}$ the entries of the matrix $N_n$ in~\eqref{def N} and
\begin{equation*}
  d_{23} := (d_b\lambda_n+\eta_b\alpha+\mu)\f{\alpha\phi'_1}r-\eta_b\,\alpha.
\end{equation*}
Furthermore, as by \eqref{def partial Q} it follows
\[
  \pa_1 \bar Q(\alpha\phi'_1+\phi_1)+\pa_2 \bar Q\beta=0\,,
\]
adding the second column, multiplied by $\f{\beta}r,$ to the first one, multiplied by $\f{\alpha\phi'_1+\phi_1}r$, we get
\[
  (1-\f\beta r)|D^{\e}|=
  \begin{vmatrix}
    N_{11}-\mu\quad&-(d_b\lambda_n+\eta_b\alpha+\mu)\quad& N_{12}\,\\[1.5ex]
    -(d_b\lambda_n+\eta_b\alpha+\mu)\f{\beta}r\quad&-(d_b\lambda_n+\eta_b\alpha+\mu)-\f{1}\e\pa_2\bar Q\quad& d_{23}\quad\\[1.5ex]
    N_{21}\quad&-\eta_v(1-\alpha)\quad&N_{22}-\mu\,
  \end{vmatrix}.
\]
Finally, subtracting the first column to the second one, multiplied by $\f{\beta}r$, we have
\begin{equation}\label{det D final}
\f\beta r(1-\f\beta r)|D^{\e}|=
\begin{vmatrix}
\quad N_{11}-\mu\quad&\quad d_{12}\quad&N_{12}\,\\[1.5ex]
d_{21}\quad&\quad -\f{1}\e\f{\beta}r\pa_2\bar Q\quad&\quad d_{23}\quad\\[1.5ex]
\quad N_{21}\quad&\quad0\quad&N_{22}-\mu\,
\end{vmatrix},
\end{equation}
with
\[
\begin{split}
d_{12}&:=\mu\Big(1-\f{\beta}r\Big)-\big(d_b\lambda_n+\eta_b\alpha\big)\f{\beta}r-N_{11},\\
d_{21}&:=-(d_b\lambda_n+\eta_b\alpha+\mu)\f{\beta}r\,.
\end{split}
\]
Thus, \eqref{det D final} rewrites as
\[
\f\beta r(1-\f\beta r)\,|D^\e(\mu)|=-\f{1}\e\,\beta\,\Big(1-\f{\beta}r\Big)\,\det(N_n-\mu\,I_2)+R(\mu)\,,
\]
where
\[
R(\mu)=-\f\beta r\Big(1-\f{\beta}r\Big)\,\mu^3+p(\mu)\,,
\]
with $p(\mu)$  a polynomial function of degree two that does not depend on $\e$. Consequently
\begin{equation}\label{eq: eigen prb final}
|D^\e(\mu)|=-\mu^3-\f r\e\det(N_n-\mu\,I_2)+\f{r^2}{\beta(r-\beta)}p(\mu)\,,
\end{equation}
with
\begin{equation}\label{eq: det N-mu I2}
\det(N_n-\mu\,I_2)=\mu^2-(tr N_n)\mu+\det N_n\,.
\end{equation}

Let $\gamma_i$, $i=1,2$ denote the eigenvalues of $N_n$ and let $\mu_i^\e$ denote the eigenvalues of $N^\e_n$, $i=1,2,3$. It has been shown that $\Re(\gamma_i)<0$ and $\Re(\mu_i^\e)<0$. Moreover, observe that $\mu_i^\e$ is a root of \eqref{eq: eigen prb final} if and only if it is a root of
\begin{equation}\label{newpol}
-\e\mu^3-r\det(N_n-\mu\,I_2)+\e \f{r^2}{\beta(r-\beta)}p(\mu)\,.
\end{equation}
Plugging in \eqref{newpol} the simple asymptotic expansion in $\e$ of $\mu_i^\e=\nu_0^i+\e\nu_1^i+\e^2\nu_2^i+\cdots$, the zero order terms gives
 $-r\det(N_n-\nu_0^i\,I_2)=0$. Therefore,
\begin{equation}\label{eq:mu i eps}
\mu_i^\e=\gamma_i+O(\e)\,,\qquad i=1,2\,,
\end{equation}
and
\[
\begin{split}
&\mu_1^\e+\mu_2^\e=\text{tr }N_n+O(\e),\\
&\mu_1^\e\mu_2^\e=\det N_n+O(\e)\,.
\end{split}
\]
On the other hand, writing $|D^\e(\mu)|=-(\mu-\mu_1^\e)(\mu-\mu_2^\e)(\mu-\mu_3^\e)$, from \eqref{eq: eigen prb final}--\eqref{eq: det N-mu I2}, we deduce the identities below
\[
\begin{split}
&\mu_1^\e+\mu_2^\e+\mu_3^\e=-\f r\e+O(1),\\
&\mu_1^\e\mu_2^\e+\mu_3^\e(\mu_1^\e+\mu_2^\e)=-\f r\e\text{tr }N_n+O(1),\\
&\mu_1^\e\mu_2^\e\mu_3^\e=-\f r\e\det N_n+O(1)\,,
\end{split}
\]
so that,
\[
\mu_3^\e=-\f r\e+O(1)\,.
\]
\section{Numerical simulations}\label{sect.simulation}
For the numerical simulations we consider the linear conversion rates
\begin{equation}\label{phi psi sect 5}
\phi(x)=x+\delta\qquad\text{and}\qquad\psi(x)=\theta x+\gamma\,,
\end{equation}
with $\delta=0.5$, $\theta=5$ and $\gamma=1$, together with the growth rates
\begin{equation}  \label{eq:growth-rates-numerics}
\eta_a=3,\qquad \eta_b=2,\qquad \eta_v=40\,.
\end{equation}
Depending on the choice of \(a\) and \(b\) we consider two cases: the $v$ extinction case
\begin{equation}
  \label{eq:extinction-numerics}
  a = 1.5,\quad b=6, \quad \Rightarrow \alpha = 1,
\end{equation}
and the coexistence case
\begin{equation}
  \label{eq:coexistence-numerics}
  a = 1.5,\quad b=8, \quad \Rightarrow \alpha < 1.
\end{equation}

In the case of the ODE system associated to the mesoscopic system \eqref{meso system} with \eqref{def react functions} and \eqref{phi psi sect 5}, the numerical solution is illustrated in Figure~\ref{ODEextv} ($\alpha=1$) and Figure~\ref{ODEcoex} ($\alpha<1$). The expected initial layer for the subpopulations $u_a^\e$ and $u_b^\e$ can be observed in Figure~\ref{ODEextvzoom} and~\ref{ODEcoexzoom} (see Remark~\ref{rm:initial layer}).

\begin{figure}[H]
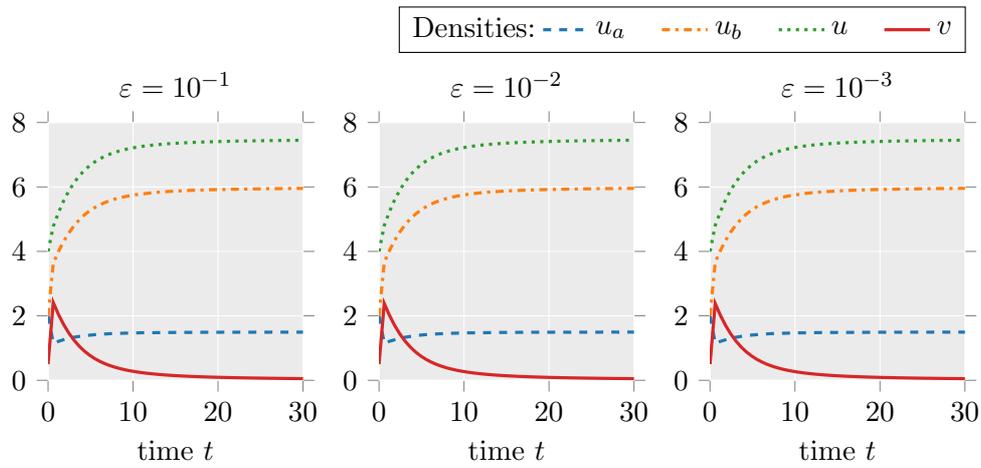

  \centering

  \caption{$\alpha=1$. Solution of the mesoscopic ODE system with
    parameters given in \eqref{phi psi sect
      5},\eqref{eq:growth-rates-numerics} and
    \eqref{eq:extinction-numerics}, for $\e=10^{-1},10^{-2},10^{-3}$
    (from left to right), with extinction of $v^\e$, and convergence
    of $u^\e=u_a^\e+u_b^\e$ towards $a+b$. Here the maximal time is
    $T=30$.}\label{ODEextv}
\end{figure}
\begin{figure}[H]
  \centering
  %
  \caption{$\alpha=1$. Zoom of the solution in Figure \ref{ODEextv}
    in a right neighbourhood of $t=0$ for $\e=10^{-1},10^{-2},10^{-3}$
    (from left to right).}\label{ODEextvzoom}
\end{figure}
\begin{figure}[H]
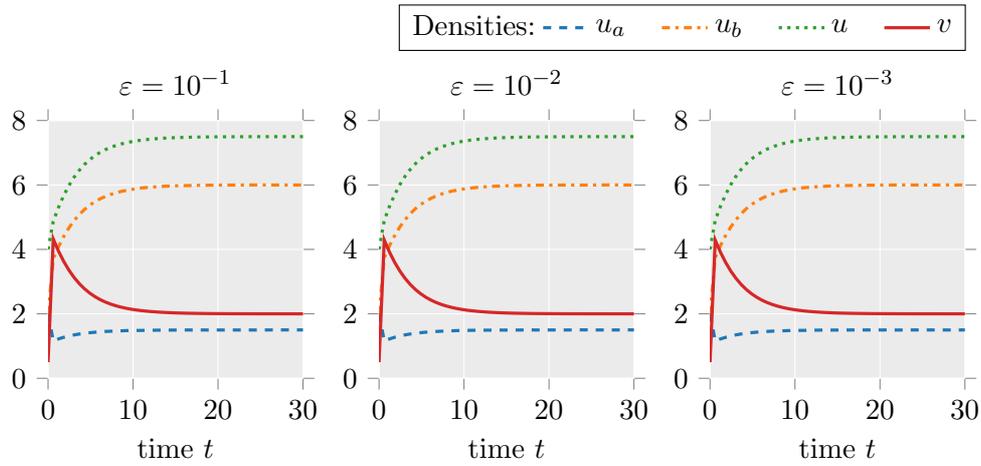

  \centering
  %
  \caption{$\alpha<1$. Solution of the mesoscopic ODE system with
    parameters given in \eqref{phi psi sect
      5},\eqref{eq:growth-rates-numerics} and
    \eqref{eq:coexistence-numerics}, for $\e=10^{-1},10^{-2},10^{-3}$
    (from left to right), with convergence of
    $(u^\e,v^\e)=(u_a^\e+u_b^\e,v^\e)$ towards
    $(a+b\alpha,b(1-\alpha))$. Here the maximal time is
    $T=30$.}\label{ODEcoex}
\end{figure}
\begin{figure}[H]
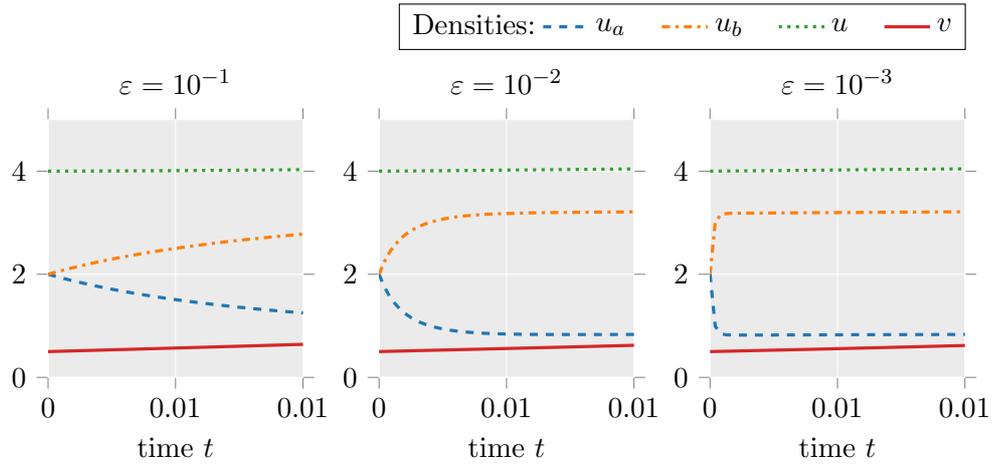

  \centering
  %
  \caption{$\alpha<1$. Zoom of the solution in Figure \ref{ODEcoex}
    in a right neighbourhood of $t=0$ for $\e=10^{-1},10^{-2},10^{-3}$
    (from left to right).}\label{ODEcoexzoom}
\end{figure}

The effect of the spatial dispersal of the species by diffusion is
shown in Figure~\ref{PDEextv} ($\alpha=1$) and Figure~\ref{PDEcoexv}
($\alpha<1$) below, in the case of the one dimensional spatial domain
$[0,1]$. Additionally, we provide videos in the supplements along with
the used code. All the parameters are kept as in the previous
computations and the diffusion coefficients are
\[
d_a=2,\qquad d_b=0.1,\qquad d_v=0.1\,,
\]
and the initial conditions has been chosen as
\[
\begin{split}
u_a^\init(x)&=\cos(4\pi x)+4\,,\qquad u_b^\init(x)=(x-1)\sin(4\pi x^2)+2\,,\\
&v^\init(x)=\cos(4\pi x)+\cos(2\pi x)+2.5\,.
\end{split}
\]
\begin{figure}[H]
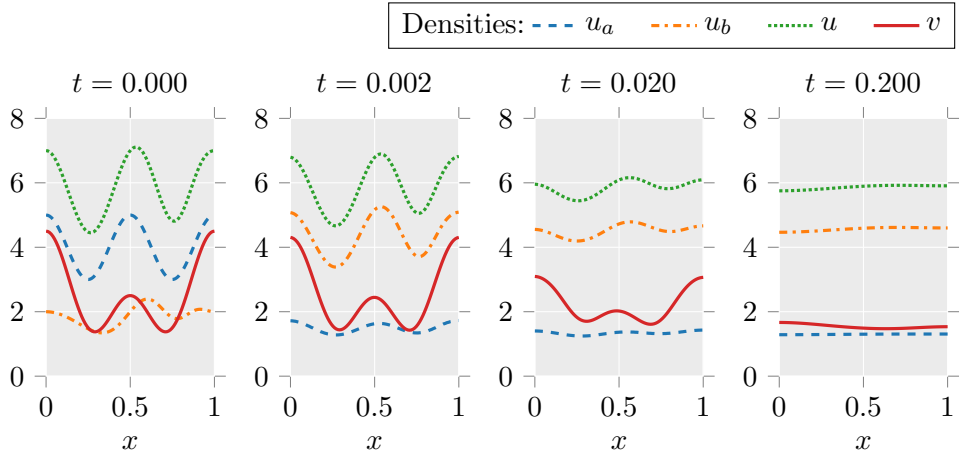

  \centering

  \caption{$\alpha=1$. Solution of the mesoscopic PDE system
    \eqref{meso system} in the extinction case.}\label{PDEextv}
\end{figure}
\begin{figure}[H]
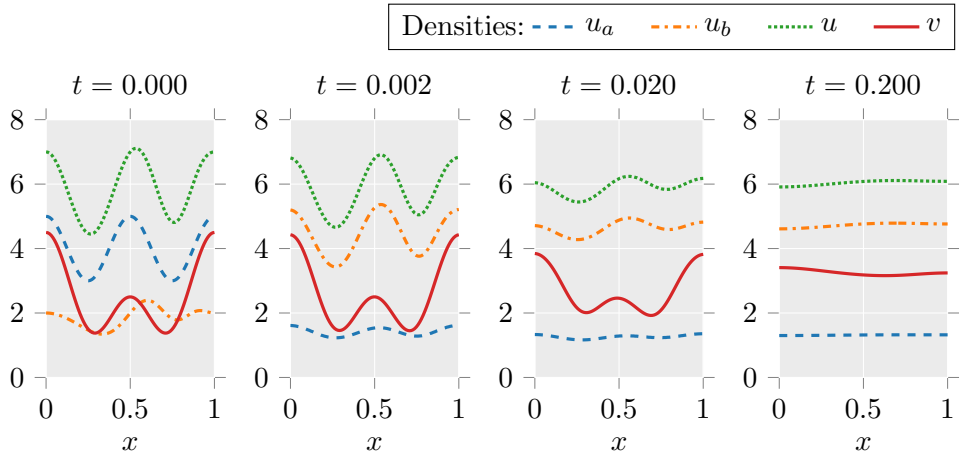

  \centering
  %
  \caption{$\alpha<1$. Solution of the mesoscopic PDE system
    \eqref{meso system} in the coexistence case.}\label{PDEcoexv}
\end{figure}
\section{Concluding remarks and discussion}\label{Discussion section}

In this paper we derive a (macroscopic) reaction-cross-diffusion system with no flux boundary conditions for two species $u$ and $v$ in competition for resources, from a (mesoscopic) diffusion system with slow and fast reaction terms and homogeneous Neumann  boundary conditions. At the mesoscopic scale, the fast reaction term governs the switching of the individuals of two sub-populations $u_a$ and $u_b$ of the species $u$, into each other, a strategy to optimise the two available ressources consumed by $u$. The individuals of the species $v$ follow a monodiet and evolve competing solely with $u_b$. As a consequence, no cross-diffusion term appears in the macroscopic equation for $v$. In other words, the reaction-cross-diffusion system is triangular.

Examples of cross-diffusion systems (triangular and not) derived by taking a fast-reaction limit can be found in  \cite{Soresina,DesTres,CJ,CDJ,DDJ,DLMT,J,bothe-pierre-rolland-2012-cross}. A different but equally popular approach for the derivation of cross-diffusion systems (not pursued in this paper) is the mean-field limit of interacting particles, e.g.~\cite{fontbona-meleard-2014-non-local-cross-diffusion,daus-desvillettes-dietert-2019-about,chen-daus-juengel-2019-rigorous,dietert-moussa-2021-persisting}.

For the mathematical analysis, a cross-diffusion term can break
general parabolic estimates including the maximum principle, so only
local existence results are usually obtained
\cite{amann-1993-nonhomogeneous}. For global existence, one needs for
instance an entropy functional to get a priori estimates and construct
a weak solution
\cite{DLMT,J,dietert-moussa-2021-persisting}. Uniqueness is a major
problem as well. The key mathematical contribution of this paper is
the identification of the entropy functional~\eqref{def energy E},
which behaves well under diffusion \emph{and} the fast-reaction term
\eqref{def Q}. By the obtained control, we are able to undergo the
fast reaction limit, to identify the limit densities $u$ and $v$ and
prove along this way that $(u,v)$ is a global in time weak
solution. Further studies for the regularity and uniqueness of the
solution of a larger class of triangular cross-diffusion systems,
including \eqref{macro}, are the objects of future works.

An interesting mathematical issue left open in this paper is the
discrepancy between the boundary conditions for the mesoscopic and
macroscopic systems: for the mesoscopic system~\eqref{meso system},
the no-flux and the Neumann boundary conditions are equivalent, but
for the limit system~\eqref{macro} we naturally obtain no-flux
boundary conditions~\eqref{BC}. Formally, the no-flux boundary
conditions~\eqref{BC} are
\begin{equation*}
  \big[d_a + (d_b-d_a) \partial_1 u_b^*\big]\, \nabla u \cdot \sigma
  + \big[(d_b-d_a)\partial_2 u_b^*\big]\, \nabla v \cdot\sigma=0
  \qquad\text{and}\qquad
  \nabla v \cdot \sigma=0,
\end{equation*}
where \(u_b^*=u_b^*(u,v)\) is the unique solution to \eqref{NLsystem}
for given \(u\) and \(v\). Differentiating the condition
\(Q(u-u_b^*(u,v),u_b^*(u,v),v)=0\) with respect to \(u\) yields
\begin{equation*}
  \partial_1 u_b^* =
  \frac{\psi\Big(\frac{u-u_b^*}{a}\Big)
    +\frac{u-u_b^*}{a}\psi'\Big(\frac{u-u_b^*}{a}\Big)}
  {\phi\Big(\frac{u_b^*+v}{b}\Big) +\frac{u_b^*}b\phi'\Big(\frac{u_b^*+v}{b}\Big)
  +\psi\Big(\frac{u-u_b^*}{a}\Big) +\frac{u-u_b^*}{a}\psi'\Big(\frac{u-u_b^*}{a}\Big)}.
\end{equation*}
Hence by our assumptions \(\partial_1 u_b^* \in (0,1)\), which implies
\(\big[d_a + (d_b-d_a) \partial_1 u_b^*\big] > 0\). Therefore, the
no-flux boundary condition is formally equivalent to the Neumann
boundary condition. It would be interesting to see whether
singularities can break this equivalence.

From a modelling point of view, we show that the competition between
$u$ and $v$ described above, can be modelled by a Lotka-Volterra
competitive type system, with competitive coefficients derived by the
population dynamics. To the best of the author's knowledge, the
meaning of the classical Lotka-Volterra competition system is only
abstract, and connecting the system coefficients to specific
situations is rarely been done. We are able to answer to this
fundamental question in the specific model of the paper. Indeed,
dropping the diffusion terms, the competition system
\eqref{macro}--\eqref{NLsystem} reads as
\begin{equation}\label{M1}
  \left\{
  \begin{aligned}
    \partial_{t}u &=\eta_au_a\left(1-\dfrac{ u_a}{a}\right) +\eta_bu_b\left(1-\dfrac{u_b+v}{b}\right),\\
    \partial_tv &=\eta_vv\left(1-\dfrac{u_b+v}{b}\right),
  \end{aligned}
  \right.
\end{equation}
and
\begin{equation}\label{m1}
  u=u_a+u_b,\quad \phi\Big(\frac{u_b+v}{b}\Big)u_b=\psi\Big(\frac{u_a}{a}\Big)u_a\,.
\end{equation}
Let
\begin{equation}\label{alphabeta}
r_a(u_a,u_b,v):=\left(1+\frac{\psi(\frac{u_a}{a})}{\phi(\frac{u_b+v}{b})}\right)^{-1}\ \text{ and }\ \ r_b(u_a,u_b,v):=\left(1+\frac{\phi(\frac{u_b+v}{b})}{\psi(\frac{u_a}{a})}\right)^{-1}.
\end{equation}
Then, it holds
\[
u_a=r_a u,\quad u_b=r_b u,\quad r_a+r_b=1,\quad 0<r_a,r_b<1,
\]
and the system can be rewritten in terms of $u$ and $v$, as the following Lotka-Volterra system
\begin{equation}\label{LVC}
\begin{cases}
  u_t =\eta_u(1-b_{11}u-b_{12}v)u,& \\
  v_t =\eta_v(1-b_{21}u-b_{22}v)v,&
\end{cases}
\end{equation}
where the competition coefficients are given by
\begin{equation}\label{a_ij}
b_{11}=\frac{\eta_ar_a^2/a+\eta_br_b^2/b}{\eta_ar_a+\eta_br_b},\quad
b_{12}=\frac{\eta_br_b/b}{\eta_ar_a+\eta_br_b},\quad
b_{21}=\frac{r_b}{b},\quad
b_{22}=\frac{1}{b},
\end{equation}
and the growth rates are
\begin{equation}\label{etas}
\eta_u=\eta_ar_a+\eta_br_b\qquad \text{ and }\qquad \eta_v\,.
\end{equation}

The fundamental difference between the classical Lotka-Volterra
competition system and our model \eqref{LVC}-\eqref{etas} is in the
solution dependency of the coefficients $r_a$ and $r_b$ in
\eqref{alphabeta} and thus of the $b_{ij}$'s. Our understanding is
that the classical Lotka-Volterra competition system with constant
coefficients $b_{ij}$ should be considered locally, where the
coefficients variation is small, while \eqref{LVC}-\eqref{etas}, can
be considered globally. In this viewpoint, we can still call the competition
modelled by system \eqref{LVC}-\eqref{etas} a strong competition if
$b_{12}/b_{11}>1$ and $b_{21}/b_{22}>1$ for all solutions.

A systematic study on the derivation of advection and cross-diffusion
terms from a given population dynamics with meaningful parameter
regimes is performed in \cite[Section 4]{Chung2019}.

Again from a modelling perspective, a motivation to consider
reaction-cross-diffusion systems is the possibility to find
instabilities due to the cross-diffusion, where a normal diffusion
cannot induce instabilities. The identification of these
cross-diffusion induced instabilities is a very active research area
\cite{IMN,wen-2013-turing,breden-kuehn-soresina-2021}. This has been
the motivation of our investigation of spatially homogeneous
stationary states and their linear stability in
\textit{Section}~\ref{sect stability}. For semi-trivial stationary
states with $v=0$, we see that the fast-reaction term can lead to
non-trivial behaviour (lack of uniqueness), see also
\cite{kim-seo-yoon-2020-asymmetric}. On the other hand, the totally
non-trivial homogeneous steady state (coexistence state) is unique and
linearly stable. Thus, the possibility of cross-diffusion induced
instability is ruled out in that case.

The existence of heterogeneous steady states of the macroscopic system
is not discarded, and it will be analyzed in a forthcoming paper.

\appendix
\section{Proof of Proposition \ref{Prop Routh M eps}}\label{Appendix A}
\begin{proof}
The Routh matrix associated to $M^\e$ writes as (see \cite{RH})
\begin{equation*}\label{Routh M eps}
R_{M^\e}\coloneqq
\begin{bmatrix}
1 & \minors M^\e\\[1.5ex]
-\text{tr}M^\e & -\det M^\e\\[1.5ex]
\dfrac{(\minors M^\e)(\text{tr}M^\e)-\det M^\e}{\text{tr}M^\e} & 0\\[1.5ex]
-\det M^\e & 0
\end{bmatrix},
\end{equation*}
with
\[
\minors M^\e \coloneqq \minor{M^\e}{11}+\minor{M^\e}{22}+\minor{M^\e}{33},
\]
and where $\minor {M^\e}{ii}$ are the following minors:
\[
\minor {M^\e}{11}\coloneqq
\begin{vmatrix}
M^\e_{22} & M^\e_{23} \\
M^\e_{32} & M^\e_{33}
\end{vmatrix},\qquad
\minor {M^\e}{22}\coloneqq
\begin{vmatrix}
M^\e_{11} & M^\e_{13} \\
M^\e_{31} & M^\e_{33}
\end{vmatrix},\qquad
\minor {M^\e}{33}\coloneqq
\begin{vmatrix}
M^\e_{11} & M^\e_{12} \\
M^\e_{21} & M^\e_{22}
\end{vmatrix}.
\]
By the Routh-Hurwitz criterion \cite{RH}, $M^\e$ is stable if and only if there are no sign variations in the first column entries of $R_{M^\e}$, i.e., if and only if $M^\e$ satisfies
\begin{equation}\label{Routh conditions}
\begin{cases}
\text{tr} M^\e<0,\\
(\minors M^\e)(\text{tr} M^\e)-\det M^\e<0,\\
\det M^\e<0.
\end{cases}
\end{equation}
From the expression of $M^\e,$ we get

\[
\text{tr } M^\e=-\eta_a-\eta_b\alpha-\eta_v(1-\alpha)-\f r\e<0,
\]
and
\begin{equation*}
\begin{split}
\minor {M^\e}{11}&=\eta_v\f{1-\alpha}\e\phi_1>0,\\
\minor {M^\e}{22}&=\eta_v(1-\alpha)\Big(\eta_a+\f \beta\e\Big)>0,\\
\minor {M^\e}{33}&=\eta_a\eta_b\alpha+\f 1\e\eta_a(r-\beta)+\f1\e\eta_b\alpha\beta>0,
\end{split}
\end{equation*}
which imply
\[
\minors M^\e>0.
\]
Furthermore,
\[
\det M^\e=\Big(-\eta_a+\f 1 \e\pa_1\bar Q\Big)\minor {M^\e}{11}-\f {\eta_v} \e\pa_1\bar Q\f{1-\alpha}{\e}(\pa_2\bar Q-\pa_3\bar Q)
=-\f{\eta_a\eta_v\phi_1}\e(1-\alpha)<0\,.
\]
It remains to check the second inequality in \eqref{Routh conditions}, that is a consequence of the previous computations and of the identity
\[
\det M^\e=-\eta_a\minor{M^\e}{11}\,.
\]
Indeed,
\[
\begin{split}
(\minors M^\e)(\text{tr }M^\e)-&\det M^\e=\big(\minor{M^\e}{11}+\minor{M^\e}{22}+\minor{M^\e}{33}\big)\text{tr }M^\e
+\eta_a\minor{M^\e}{11}\\
=&\big(\minor{M^\e}{22}+\minor{M^\e}{33}\big)\text{tr }M^\e-\minor{M^\e}{11}\Big(\eta_b\alpha+\eta_v(1-\alpha)+\f r\e\Big)<0\,.
\end{split}
\]
Thus, $M^\e$ is stable for all $\e>0$.

Concerning the matrix $N^\e$, we define the quantities
\begin{equation}\label{def:Di}
D_1\coloneqq d_a+d_b+d_v>0\,,\quad D_2\coloneqq d_ad_v+d_bd_v+d_ad_b>0\,,\quad D_3\coloneqq d_ad_bd_v\,,
\end{equation}
and
\begin{equation}\label{def:ABC}
\begin{split}
&A\coloneqq d_a (M^\e_{22}+M^\e_{33})+d_b (M^\e_{11}+M^\e_{33})+d_v(M^\e_{11}+M^\e_{22})<0,\\
&B\coloneqq d_bd_v M^\e_{11}+d_ad_v M^\e_{22}+d_a d_b\, M^\e_{33} <0,\\
&C\coloneqq d_a\,\minor{M^\e}{11}+d_b\minor{M^\e}{22}+d_v\minor{M^\e}{33}>0\,.
\end{split}
\end{equation}
Thus, using the previous computations, we obtain
\[
\text{tr }N^\e=\text{tr }M^\e-D_1\lambda_n<0\,,
\]
\[
\minors N^\e=\minors M^\e+ D_2\,\lambda_n^2- A\,\lambda_n>0,
\]
and
\[
\det N^\e =\det M^\e-D_3\,\lambda_n^3+B\,\lambda_n^2-C\,\lambda_n<0\,.
\]
To conclude, it remains to check the sign of the quantity below:
\begin{equation*}
\begin{split}(\minors N^\e)(\text{tr }N^\e)-\det N^\e&=(\minors M^\e)(\text{tr }M^\e)-\det M^\e\\
&\quad +\lambda_n^3(-D_1D_2+D_3)+\lambda_n^2(D_2\text{tr}\,M^\e+AD_1-B)\\
&\quad+\lambda_n(-D_1\minors M^\e-A\,\text{tr}\,M^\e+C)\,.
\end{split}
\end{equation*}
The latter is indeed strictly negative, using again the negativity of the entries of $M^\e$, the positivity of the minors $\minor{M^\e}{ii}$, definitions \eqref{def:Di} and \eqref{def:ABC} and
\[
-D_1D_2+D_3<0\,,\qquad AD_1-B<0\,,\qquad -D_1\minors M^\e+C<0\,.
\]
Then, by the Routh-Hurwitz criterion again, $N^\e$ is stable for all strictly positive~\(\e\).
\end{proof}
\bigskip
\begin{acknowledgment}
The authors warmly thank Laurent Desvillettes for the fruitful discussions about the model and his useful suggestions. This international collaboration was made possible
through the International Research Network (IRN) ``ReaDiNet'' financed by CNRS, France, and Korea Advanced Institute of Science and Technology (KAIST), Korea.
\end{acknowledgment}
\bibliographystyle{spmpsci}
\bibliography{bcdk}

\noindent Email addresses: \\
Elisabetta Brocchieri : elisabetta.brocchieri@univ-evry.fr\\
Lucilla Corrias : lucilla.corrias@univ-evry.fr\\
Helge Dietert : helge.dietert@imj-prg.fr\\
Yong-Jung Kim : yongkim@kaist.edu

\bigskip
\noindent ${^1}$ Laboratoire de Math\'ematiques et Mod\'elisation d'Evry (LaMME),\\
UEVE and UMR 8071, Paris Saclay University \\
23 Bd. de France, F--91037 Evry Cedex, France\\
\noindent ${^2}$
Universit\'e de Paris and Sorbonne Universit\'e, CNRS,
Institut de Math\'ematiques de Jussieu-Paris Rive Gauche (IMJ-PRG),
F-75013, Paris, France\\
Currently on leave and working at\\
Institut f\"ur Mathematik, Universit\"at Leipzig, D-04103 Leipzig, Germany\\
\noindent ${^3}$ Department of Mathematical Sciences,\\
Korea Advanced Institute of Science and Technology,\\
291 Daehak-ro, Yuseong-gu, Daejeon, 34141, Korea\\
\end{document}